\documentclass{mcom-l}

\usepackage{graphicx,xcolor} 
\usepackage{amssymb,amsmath}
\usepackage{algorithm}
\usepackage{algpseudocode}

\newtheorem{theorem}{Theorem}[section]

\newtheorem{corollary}[theorem]{Corollary}
\newtheorem{lemma}[theorem]{Lemma}

\theoremstyle{definition}
\newtheorem{definition}[theorem]{Definition}
\newtheorem{example}[theorem]{Example}

\theoremstyle{remark}
\newtheorem{remark}[theorem]{Remark}

\numberwithin{equation}{section}

\newcommand{\op}[1]{{\mathcal #1}}
\newcommand{\mat}[1]{{\mathbf #1}}
\newcommand{\interp}{{\mathcal I}}

\newcommand{\eqdef}{{\stackrel{\mathrm{def}}{=}}}
\newcommand{\innprd}[2]{\left< #1 , #2 \right>}
\newcommand{\linnprd}[2]{\left( #1 , #2 \right)}

\def\R{\mathbb{R}}

\def\N{\mathbb{N}}

\begin{document}

\title[Sufficient conditions for strong discrete maximum principles]
{Sufficient conditions for strong discrete maximum 
  principles in finite element solutions of linear and semilinear elliptic equations}


\author{Andrei Dr{\u{a}}g{\u{a}}nescu}
\address{Department
    of Mathematics and Statistics, University of Maryland, Baltimore
    County, 1000~Hilltop Circle, Baltimore, Maryland 21250}
\curraddr{}
\email{draga@umbc.edu}
\thanks{The first author was supported in part by NSF Award 2409951.}

\author{L.~Ridgway~Scott}
\address{Department of Mathematics, University of Chicago, Chicago, 
    5734 S. University Avenue, Chicago, Illinois, 60637}
\curraddr{}
\email{ridg@uchicago.edu}
\thanks{}


\date{}

\dedicatory{}

\begin{abstract}
We introduce a novel technique for proving global strong 
discrete maximum principles for  finite element discretizations
of linear and semilinear elliptic equations for cases when the common, 
matrix-based sufficient conditions are  not satisfied. The basic argument
consists of extending the strong form of discrete maximum principle from 
macroelements to the entire domain
via a connectivity argument. The method is applied to
discretizations of elliptic  equations with certain
pathological meshes, and to semilinear elliptic  equations.
\end{abstract}

\maketitle

\section{Introduction}
\label{sec:intro}

The preservation of qualitative properties constitutes a central theme
in the design of numerical methods for partial differential
equations. Among those properties, maximum principles have captured
the attention of many generations of numerical analysts, as they play
an essential role in ensuring that solutions maintain their physical
relevance. For example, it is not only desired, but sometimes critical
that quantities representing concentrations lie in the interval [0,1],
that computed densities are positive, or that fluxes across interfaces 
have the correct sign.

In this work we focus on discrete maximum principles (DMPs) for finite
element solutions of linear and semilinear elliptic equations. A short
and particular formulation of the DMP is that the maximum of a discrete subharmonic function 
cannot be achieved in the interior of its domain unless the function is constant.
At the discrete level we distinguish between the \emph{local} DMP, which refers to 
the DMP being satisfied on the union of elements with a common vertex, and 
the \emph{global} DMP, which refers to the entire domain; verifying
the latter is the ultimate goal, but also presents a greater challenge.
A long list of works~\cite{MR375802, MR443377, MR1654022, MR1803125, MR2121074, MR2226992, MR2392463, 
MR2520861, MR2607805, MR2501643, 
MR2914278, MR3425305, MR4848416} (to cite only a few)
was devoted to studying conditions under which appropriate forms of  
the global DMP hold for linear and nonlinear elliptic equations. 
Many more references can be found in the review article~\cite{MR4704682} and
the recent monograph~\cite{MR4900691}. A common element for all these articles is 
the hypothesis that certain key matrices have nonpositive off-diagonal elements;
for the case of the linear Poisson equation, this reduces to the necessity for 
the stiffness matrix to be an $M$-matrix (see~\cite{MR1753713} for definition).
Not only is this condition restrictive on the mesh (see Definition 2.2 in~\cite{MR4704682}), 
but it is equivalent to the local DMP to be satisfied
around each vertex. Hence, it is fair to say that most of the aforementioned works
use various techniques to globalize the DMP, after essentially assuming the local DMP holds everywhere.
However, in~\cite{MR2085400} (Section 6) it is shown that the global DMP can hold on certain meshes where 
local DMPs do not hold. Therefore, the nonpositivity of the off-diagonal entries is not a necessary condition
for the  global DMP (see also~\cite{MR4900691}, Example 6.10, p. 159). 

Moreover, an example is given~\cite{MR2085400} where the global DMP fails, even as the 
mesh size converges to zero. It is notable how challenging it is to find an example where the finite element 
spaces have good approximation properties, but the
global DMP fails; in fact the global DMP seems to  hold for many practical situations.
Hence, it is fair to say that there is a gap in the literature between the known sufficient and the 
necessary conditions for the global DMP to hold. In this paper we are providing a set of conditions
that aim to bridge this gap. We also note that approximation alone allows proving a weaker form of
the DMP~\cite{MR551291, MR4166450}.

The main contribution in this article is to provide a novel technique for proving a global 
\emph{strong} DMP (sDMP) for several classes of elliptic equations. The main ingredient  {\color{black}consists} 
of extending the  sDMPs from macroelements to the entire domain using a connectivity argument;
nonpositivity of the off-diagonal entries of the stiffness matrix is not assumed to hold everywhere.
The general theorems are applied to linear elliptic equations that include ``defects'' (edges
that lead to positive off-diagonal entries in the stiffness matrix), degenerate meshes, as well as semi-linear elliptic equations.

This paper is organized as follows. After introducing the maximum principles 
in Section~\ref{sec:problemform}, we present the main results in abstract form in Section~\ref{sec:mainres}. 
These are applied to linear elliptic equations in Section~\ref{sec:linelliptic}, where 
we also connect the new technique with classical results.
In Section~\ref{ssec:nonmonotedge} we apply our framework to problems with mesh defects, and
in Section~\ref{sec:semilinear} we  prove the sDMP for 
a class of semilinear elliptic equations; the latter results may not be completely new, 
but they further showcase the wide scope of our method's applicability. Some conclusions are formulated in Section~\ref{sec:conclusion}.


\section{Motivation and problem formulation}
\label{sec:problemform}
In this section we introduce the  model problems of interest and the classical maximum principles, on which
we model their discrete counterparts in Section~\ref{sec:mainres}.
\subsection{The continuous model problems and their maximum principles}
\label{ssec:contproblemform}
Let $D\subset \R^d$ ($d=2, 3$) be a polygonal or polyhedral domain, 
and consider the monotone semilinear elliptic boundary value problem
\begin{subequations}
\begin{equation}
\label{eq:slcontdef}
-\sum_{i,j=1}^d \partial_i(a_{ij}(x)\partial_j u(x)) + c(x,u(x))  = f(x)\ \ \mathrm{in} \ \  D,
\end{equation}
\begin{equation}
\label{eq:slcontBC}
u|_{\partial D} =  g,
\end{equation}
\end{subequations}
with {\color{black}$a_{ij}=a_{ji}\in C^{0,1}(\overline{D})$} for $1\le i, j \le d$, and 
\begin{eqnarray}
\label{eq:ellop}
\underline{a}|v|^2\le \sum_{i,j=1}^d a_{ij}(x) v_i v_j \le \overline{a}|v|^2,\ \ \forall v\in \R^d
\end{eqnarray}
for some constants {\color{black}$0<\underline{a}\le \overline{a}$}. Assume the reaction term $c:D\times \R \to \R$ 
satisfies the following conditions:
\begin{align}
\label{eq:Cnondecreasing}
&\forall x\in D,\  c(x,0)=0, \ \ \mathrm{and}\ \ c(x,\cdot)\ \ \mathrm{is\  nondecreasing;}\\
&\label{eq:CinLP}
{\color{black}c(\cdot,0)\in L^{\frac{p}{2}}(D)\ \mathrm{for\ some\ }p>d;}\\
&\label{eq:cincrease}
\exists L_c>0, \ \ \forall x\in D,\  \forall u, v\in \R,\ |c(x,u)-c(x,v)| \le  L_c |u-v|.
\end{align}
This includes the linear case when 
\begin{align}
\label{eq:Creaclinear}
&c(x,u) = \tilde{c}(x)\: u
\end{align}
for some nonnegative function 
{\color{black}$\tilde{c}\in L^{\infty}(D)$}.
{\color{black}For existence, uniqueness, and regularity results for~\eqref{eq:slcontdef}-\eqref{eq:slcontBC} 
see~\cite{troltzsch2010optimal} and the references therein.}

For a function $u$ we define $u^+ = \max (u,0)$, and $u^- = -\min(u,0)$. Note that $u^+,u^-\ge 0$ and $u=u^+-u^-$. 
The continuous problem~\eqref{eq:slcontdef}-\eqref{eq:slcontBC} is 
known~\cite{gilbarg1977elliptic,evans2022partial} to satisfy the following strong maximum principles:
\begin{theorem}
\label{thm:cont_max_princLin}
Assume $c\equiv 0$ in~\eqref{eq:slcontdef}, $a_{ij}$ are continuously differentiable, and
$u\in C^2(D) \cap C^1(\overline{D})$ solves~\eqref{eq:slcontdef}-\eqref{eq:slcontBC}.\\
\textnormal{(i)} If $f\ge 0$, then
\begin{equation}
\label{eq:linmaxprincpos} 
\min_{\overline{D}}u \ge \min_{\partial D} u.
\end{equation}
In addition, if $u$ attains a  minimum over $\overline{D}$ at $x_0\in  D$ (i.e., an interior point), then
$u$ is constant.\\
\textnormal{(ii)} If $f\le 0$, then
\begin{equation}
\label{eq:linmaxprincneg} 
\max_{\overline{D}}u \le \max_{\partial D} u.
\end{equation}
In addition, if $u$ attains a  maximum over $\overline{D}$ at $x_0\in  D$ (i.e., an interior point), then
$u$ is constant.
\end{theorem}
\begin{theorem}
\label{thm:cont_max_princSL}
Assume $c$ is given by~\eqref{eq:Creaclinear}, with $c$ and $a_{ij}$ being continuously differentiable,
and $u\in C^2(D) \cap C^1(\overline{D})$ solves~\eqref{eq:slcontdef}-\eqref{eq:slcontBC}.\\
\textnormal{(i)} If $f\ge 0$, then
\begin{equation}
\label{eq:slmaxprincpos} 
\min_{\overline{D}}u \ge -\max_{\partial D} u^-.
\end{equation}
In addition, if $u$ attains a nonpositive minimum over $\overline{D}$ at $x_0\in  D$ (i.e., an interior point), then
$u$ is constant.\\
\textnormal{(ii)} If $f\le 0$, then
\begin{equation}
\label{eq:slmaxprincneg} 
\max_{\overline{D}}u \le \max_{\partial D} u^+.
\end{equation}
In addition, if $u$ attains a nonnegative maximum over $\overline{D}$ at $x_0\in  D$ (i.e., an interior point), then
$u$ is constant.
\end{theorem}
%
In this work we analyze a set of conditions under which maximum principles similar to 
Theorems~\ref{thm:cont_max_princLin}-\ref{thm:cont_max_princSL} hold for finite element discretizations 
of~\eqref{eq:slcontdef}-\eqref{eq:slcontBC}. In light of this goal, we note that 
for the continuous problem the strong maximum principle holds on any sufficiently 
regular subdomain $E\subseteq D$; for the discretized problem such maximum principles may hold
on the full domain, but not on subdomains, and {\color{black}vice versa}.

We also remark that the maximum principles for the continuous problem
are used in this work only to serve as models for  their discrete
counterparts; the continuous problem and its properties does not play
any role in the analysis of the discrete  maximum principle that we
present in the next sections.

\subsection{Finite element discretization}
\label{ssec:discproblemform}
The weak form of~\eqref{eq:slcontdef}-\eqref{eq:slcontBC}
reads: given $u^b\in H^1(D)$ so that $u^b|_{\partial D}=g$ and 
$f\in H^{-1}(D)$,  find $u\in H^1(D)$ so that \mbox{$u^0=u-u^b\in H_0^1(D)$} satisfies
\begin{eqnarray}
\label{eq:weakellcont}
a_D(u^0+u^b,v) + \linnprd{c(\cdot, u^0+u^b)}{v}_{D} =\innprd{f}{v}_D,\ \ \forall v\in H_0^1(D)\ ,
\end{eqnarray}
where $\innprd{\cdot}{\cdot}_D$ is dual pairing, $\linnprd{\cdot}{\cdot}_D$ is the $L^2$-inner product on $D$, and
\begin{eqnarray}
\label{eq:bilform}
a_D(u,v)=\sum_{i,j=1}^d \int_D a_{ij}\:\partial_i u \:\partial_j v,\ \ \forall u, v\in H^1(D).
\end{eqnarray}
Let $\op{T}_h$ be a triangulation of $\overline{D}$ with vertices 
$(P_i)_{1\le i\le N}$, and consider the space $V^h(\overline{D})$
of continuous piecewise linear functions with respect to $\op{T}_h$, and 
$V_0^h(\overline{D})=\{\varphi\in V^h(\overline{D})\ : \varphi|_{\partial D} = 0\}$.
{\color{black}Let} $(\varphi_i)_{1\le i\le N}$ the associated standard nodal basis in 
$V^h(\overline{D})$ (including boundary nodes). 
Denote by $\interp_h:C(\overline{D})\to V_h(\overline{D})$ be the nodal interpolation operator defined by
\begin{align*}
\interp_h u = \sum_{i=1}^N u(P_i) \varphi_i.
\end{align*}
The finite element solution of~\eqref{eq:slcontdef}-\eqref{eq:slcontBC} {\bf on $D$}
reads:
 find $u_h\in V^h(\overline{D})$ of the form $u_h = u_h^0+ u_h^b$
with $u_h^0\in V_0^h(\overline{D})$, and $u_h^b = \interp_h  u^b$, so that 
\begin{eqnarray}
\label{eq:weakelldisc}
a_D(u_h^0+u_h^b,v) + \linnprd{c(\cdot,u_h^0+u_h^b)}{v}_D=\innprd{f}{v}_D,\ \ \forall v\in V_0^h(\overline{D}).
\end{eqnarray}
A set $E \subseteq D$ is called a \emph{discrete subdomain} (of $D$)
if $E$ is a union of {\color{black}simplices} of $\op{T}_h$.
When solving the problem on a discrete subdomain $E$, then $E$ must replace~$D$ everywhere in~\eqref{eq:weakelldisc}.
Due to the potential nonlinearity of the reaction term $c$ in the second argument,
the term $\linnprd{c(\cdot, u_h^0+u_h^b)}{v}_D$ is  replaced with
a cubature, preferably one that involves only function values at the vertices. This is equivalent to
interpolating the aforementioned term prior to integration. More precisely,~\eqref{eq:weakelldisc} 
can be solved  in practice using the following formulation:
\begin{eqnarray}
\label{eq:weakelldiscinterp}
a_D(u_h^0+u_h^b,v) + \linnprd{\interp_h c(\cdot,u_h^0+u_h^b)}{v}_D=\innprd{f}{v}_D,\ \ \forall v\in V_0^h(\overline{D}).
\end{eqnarray}
The formulations~\eqref{eq:weakelldisc} and~\eqref{eq:weakelldiscinterp} are certainly equivalent if 
$c(x,u)=\tilde{c} u$ with $\tilde{c}\ge 0$  constant, but in general {\color{black}they} are not.
{\color{black}Cf.~\cite{MR2009948}, both the continuous variational problem~\eqref{eq:weakellcont} as well as its
discrete version~\eqref{eq:weakelldisc} are well-posed.}
In preparation for the results in Section~\ref{sec:mainres}, we highlight the following localization property of the finite element 
solutions:
\begin{lemma}
\label{lma:restriction}
Let $E\subseteq D$ be a discrete subdomain of $D$. If $u^D_h$ is the solution of~\eqref{eq:weakelldisc} 
{\color{black}or~\eqref{eq:weakelldiscinterp}} on ${D}$ 
and $u^E_h$ solves~\eqref{eq:weakelldisc} {\color{black}or~\eqref{eq:weakelldiscinterp}, respectively,} on ${E}$ 
with boundary condition
$u^E_h|_{\partial E} = u^D_h|_{\partial E}$, then $u^E_h = u^D_h|_{E}$. 
\end{lemma}
\begin{proof}
{\color{black}The argument is the same for both formulations, so we focus on the former.}
First note that $u^D_h|_E$ satisfies the required boundary condition on $E$ (in a trivial manner). 
Consider the subset of nodal basis functions supported in $\overline{E}$, that is,  
$N_E = \{i\: :\: \mathrm{supp}(\varphi_i)\subseteq \overline{E}\}$. 
Since~\eqref{eq:weakelldisc} holds for all $v=\varphi_i$ with  
$i\in N_E$, all the integrals in~\eqref{eq:weakelldisc} are, in effect,  restricted to $E$. 
Also note that $V_0^h(\overline{E})$ is generated by the set $\{\varphi_i|_E\: :\: i\in N_E\}$.
Hence, $u^D_h|_E$ satisfies~\eqref{eq:weakelldisc} with the set $E$ replacing~$D$.
\end{proof}
Essentially, Lemma~\ref{lma:restriction} states that the restriction of a finite element 
solution to a discrete subdomain $E$ is a finite element solution on $E$ with appropriate boundary conditions.

\section{Abstract strong discrete maximum principles}
\label{sec:mainres}
\subsection{Abstract solution operators}
\label{ssec:abst_solops}
Denote  $V^h(\partial D) = \{\varphi|_{\partial D}\: :\: \varphi\in V^h(\overline{D})\}$.
We say that $f\in (V_0^h(\overline{D}))^*$ is \emph{nonnegative} (or \emph{nonpositive}, respectively),
if $\innprd{f}{\varphi_i}_D\ge 0$ (respectively,  $\innprd{f}{\varphi_i}_D\le 0$), 
for all nodal basis functions $\varphi_i$, $i=1,\dots, N$. Furthermore, if $E\subseteq D$ is a discrete subdomain,
we denote by $f_E$ the natural restriction of  $f$ to $(V_0^h(\overline{E}))^*$, which is defined  by
$\innprd{\varphi}{f_E} = \innprd{\varphi^D}{f}$, where $\varphi^D$ is the extension with $0$ of $\varphi\in V_0^h(\overline{E})$
to  $\overline{D}$. {\color{black}Clearly}, if $f$ is nonnegative/nonpositive, then $f_E$ is also nonnegative/nonpositive, respectively.
The objects for which we prove  DMPs are the following abstract solution operators.
\begin{definition}
\label{def:soloperators}
Assume for each discrete  subdomain $E\subseteq D$ we have an operator
$\op{S}^E_h: (V_0^h(\overline{E}))^*\times V^h(\partial E) \to V^h(E)$. We say that the
family $(\op{S}^E_h)_{E\subseteq D}$ is a \emph{consistent family of solution operators} if for all
discrete subdomains  $E\subseteq F \subseteq D$ and any $(f,u_h^b) \in (V_0^h(\overline{F}))^*\times V^h(\partial F)$
we have
\begin{equation}
\label{eq:consistency}
\op{S}^E_h(f_E,u_h|_{\partial E}) = u_h|_{E},
\end{equation}
where $u_h = \op{S}^F_h(f,u^b_h)$.
\end{definition}
Note that Lemma~\ref{lma:restriction} implies that the family of  solution operators
$\left(\op{S}^E_h\right)_{E\subseteq D}$ of the discrete semilinear elliptic equation~\eqref{eq:weakelldisc}, 
defined by 
\begin{align*}
\op{S}^E_h(f,u_h^b) \ \eqdef\   u_h = u_h^0 + u_h^b,
\end{align*}
is consistent, i.e., it satisfies Definition~\ref{def:soloperators}, where $D$ is replaced by $E$ in~\eqref{eq:weakelldisc}. 

We now describe the DMPs for which we provide sufficient conditions in Sections~\ref{ssec:mainresults}--\ref{ssec:moreDMPs}. 
Note that these definitions refer to a solution operator associated with {\bf a single} discrete subdomain, 
not to  the entire family of consistent solution operators. In the interest of the presentation, we will focus our 
discussion on the case when $f$ is nonnegative.
\begin{definition}
\label{def:sdmpL}
Assume $f\in (V_0^h(\overline{D}))^*$ is nonnegative.\\
\textnormal{(i)} We say that a solution operator satisfies the \emph{A-version of the weak discrete maximum principle} 
(wDMP-A) if 
\begin{equation}
\label{eq:ldmppos} 
\min_{\overline{D}}\op{S}^D_h(f,u_h^b) \ge \min_{\partial D} u_h^b,\ \ \forall u_h^b\in V^h(\partial D).
\end{equation}
\textnormal{(ii)} We say that a solution operator satisfies the \emph{A-version of the strong discrete maximum principle} 
(sDMP-A) if~\eqref{eq:ldmppos} holds and, in addition, 
if $u_h = \op{S}^D_h(f,u_h^b)$ attains a  {\bf global} minimum over $\overline{D}$ at an interior vertex in $D$, then
$u_h$ is constant.
\textnormal{(iii)} We say that a solution operator satisfies the \emph{B-version of the weak discrete maximum principle} 
(wDMP-B) if
\begin{equation}
\label{eq:sldmppos} 
\min_{\overline{D}}\op{S}^D_h(f,u_h^b) \ge -\max_{\partial D} (u_h^b)^-,\ \ \forall u_h^b\in V^h(\partial D).
\end{equation}
\textnormal{(iv)} We say that a solution operator satisfies the \emph{B-version of the strong discrete maximum principle} 
(sDMP-B) if~\eqref{eq:sldmppos} holds and, in addition, if $u_h = \op{S}^D_h(f,u_h^b)$ attains a {\bf global 
nonpositive} minimum over $\overline{D}$ at an interior vertex in $D$, then
$u_h$ is constant.\\
\textnormal{(v)} We say that a solution operator satisfies the \emph{restricted version of the strong discrete maximum principle} 
(sDMP-R) if it satisfies sDMP-A whenever $u_h = \op{S}^D_h(f,u_h^b)$ is nonnegative.
\end{definition}
While the A and B versions of the DMP need no additional motivation, as they mimic the continuous counterparts, sDMP-R
is introduced as a technical tool for proving  Theorem~\ref{thm:wdmpa}, as it 
is related to the case when $c$ is given by~\eqref{eq:Creaclinear} with $\tilde{c}\le 0$. We remark  that the main
path to proving DMPs goes through proving the strong versions sDMP-A and sDMP-B, while the weak versions will be proved using 
a perturbation argument. 
{\color{black}We also note that sDMP-R is weaker  than sDMP-A, since 
it is described by the same properties, but applies only under restrictive conditions. Clearly, sDMP-A implies wDMP-A. 
The relationship between sDMP-A and sDMP-B is discussed below.}

\begin{remark}
\label{rem:A-implies-B}
Note that sDMP-A is a stronger condition than sDMP-B. Indeed, assume $\op{S}^D_h$ satisfies sDMP-A.
Let $f\in (V_0^h(\overline{D}))^*$ be nonnegative, and $u_h^b\in V^h(\partial D)$ be arbitrary. 
Denote by $L  = \min u_h^b$, and 
$u_h = \op{S}^D_h(f,u_h^b)$. If $L\ge 0$, then $(u_h^b)^- = 0$. Cf.~\eqref{eq:ldmppos} we have 
\begin{align}
\label{eq:l-implies-sl}
\min_{\overline{D}} u_h \ge L  \ge 0 = -\max (u_h^b)^-.
\end{align}
Furthermore, if $ u_h(x) = -\max (u_h^b)^-$ for some interior point $x\in D$, then~\eqref{eq:l-implies-sl} implies that
\begin{align*}
u_h(x) = \min_{\overline{D}} u_h = L = 0.
\end{align*}
By Definition~\ref{def:sdmpL} (ii), $u_h$ is constant on $\overline{D}$.
If $L<0$, then there exists $x\in \partial D$ so that $u_h^b(x) = L<0$, so $(u_h^b)^-(x) = -L$. Hence, $L=-\max (u_h^b)^-$; therefore,
\begin{align*}
\min_{\overline{D}} u_h \ge L = -\max (u_h^b)^-.
\end{align*}
Again, if equality holds above, then $u_h$ is constant on $\overline{D}$.
\end{remark}

\subsection{Main results}
\label{ssec:mainresults}
Perhaps the most desired property  for families of discrete solution operators to have is that {\bf every member} of the family
satisfies some form of a DMP. This property is expressed in the literature as the solution operator
satisfying both  a {\bf local} and a {\bf global} DMP. We will see later how this property is related to the classical angle condition
in case of the Poisson equation.
However, in this work we are focused on sufficient conditions under which the 
solution operator of the {\bf entire domain} $D$ satisfies a certain DMP, meaning we are primarily targeting  global DMPs. 
The following result describes sufficient conditions for strong DMPs to hold on the entire domain.

\begin{theorem}
\label{thm:DMP}
Let $(\op{S}^E_h)_{E\subseteq D}$ be a {consistent family of solution operators},
and the triangulation $\op{T}_h$ be so that the graph of the interior vertices {\color{black}is connected}, 
and that every boundary vertex is adjacent to an interior vertex. \\
{\bf (A)} Assume that for every interior vertex 
$Q\in D$ there exists a discrete subdomain $E_Q\subseteq D$ so that $Q\in Int(E_Q)$, and the following conditions
hold: for every $(f,u_h^b) \in (V_0^h(\overline{E}_Q))^*\times V^h(\partial E_Q)$ with $f$ nonnegative, if
$u_h=\op{S}^{E_Q}_h(f,u_h^b)$, then 
\begin{itemize}
\item[\textnormal{(A1)}] $u_h(Q) \ge \min u_h^b$ and 
\item[\textnormal{(A2)}] if $u_h(Q) = \min u_h^b$, then $u_h$ is constant on $\overline{E}_Q$.
\end{itemize}
Then $\op{S}^{D}_h$ satisfies sDMP-A.\\
{\bf (B)} Assume that for every interior vertex 
$Q\in D$ there exists a discrete subdomain $E_Q\subseteq D$ so that $Q\in Int(E_Q)$, and the following conditions
hold: for every $(f,u_h^b) \in (V_0^h(\overline{E}_Q))^*\times V^h(\partial E_Q)$ with $f$ nonnegative, if
$u_h=\op{S}^{E_Q}_h(f,u_h^b)$, then 
\begin{itemize}
\item[\textnormal{(B1)}] $u_h(Q) \ge -\max (u_h^b)^-$ and 
\item[\textnormal{(B2)}] if $u_h(Q) = -\max (u_h^b)^-$, then $u_h$ is constant on $\overline{E}_Q$.
\end{itemize}
Then $\op{S}^{D}_h$ satisfies sDMP-B.\\
{\bf (R)} Assume that for every interior vertex 
$Q\in D$ there exists a discrete subdomain $E_Q\subseteq D$ so that $Q\in Int(E_Q)$, and the following conditions
hold: for every $(f,u_h^b) \in (V_0^h(\overline{E}_Q))^*\times V^h(\partial E_Q)$ with $f$ nonnegative, if
$u_h=\op{S}^{E_Q}_h(f,u_h^b)$ is nonnegative, then \textnormal{(A1)} and \textnormal{(A2)} hold.
Then $\op{S}^{D}_h$ satisfies sDMP-R.\\
\end{theorem}
\begin{proof}
Let $(f,u_h^b) \in (V_0^h(\overline{D}))^*\times V^h(\partial D)$ with $f$ nonnegative, and $u_h=\op{S}^{D}_h(f,u_h^b)$. 
Denote by $m=\min_{i=1}^N u_h(P_i)$; this is the minimum among the values at the
vertices, but for piecewise linear functions it coincides with the global minimum of $u_h$ on $\overline{D}$.
Consider the set of vertices
$$\op{C}_m=\{Q\ \textnormal{vertex\ in\ }\ \op{T}_h\: : \: u_h(Q)=m\}.$$ 
The arguments vary only slightly between {\bf (A)}, {\bf (B)}, and {\color{black}{\bf (R)}}.

For {\bf (A)}, if $\op{C}_m \subseteq \partial D$, then $u_h$ attains
its minimum {\bf only} on $\partial D$, and~\eqref{eq:ldmppos} holds in a strict sense. Otherwise, there exists $Q\in\op{C}_m \cap Int(D)$. 
It remains to show that $u_h$ is constant, then the conclusion follows.
Let $E_{Q}$ be as in the hypothesis. Since the minimum of $u_h$ on $\overline{E}_{Q}$ is achieved at $Q$ (because it is 
a global minimum of $u_h$), and by~\eqref{eq:consistency}
{\color{black}
\begin{eqnarray*}
u_h|_{E_Q} = \op{S}^{E_Q}_h(f_{E_Q},u_h|_{\partial E_Q}),
\end{eqnarray*}
}
it follows from  property (A2)
that $u_h$ is constant on $\overline{E}_Q$.
Now let $\widehat{Q}\in Int(D)$ be an arbitrary interior vertex, and let $Q=Q_0, Q_1, \dots, Q_r=\widehat{Q}$ be a 
sequence of adjacent interior vertices connecting $Q$ to $\widehat{Q}$, that is, 
$Q_{i-1}$ is adjacent to $Q_{i}$ for $i=1, \dots, r$. Since $Q_i\in Int(E_{Q_i})$
and $Q_{i-1}$ is adjacent to $Q_i$, it follows that $\{Q_{i-1},Q_i\} \subset E_{Q_{i-1}} \cap  E_{Q_{i}}$. Assume $u_h(Q_k)=m$
for some $0\le k \le r-1$. Since $m$ is the global minimum, the same argument used for $k=0$ {\color{black}shows} that $u_h$ is constant on 
$\overline{E}_{Q_{k}}$. Hence $u_h(Q_{k+1})=m$. By induction over $k$ we get $u_h(Q_r)=u_h(Q_0)=m$. So $u_h(\widehat{Q})=m$ for all 
interior vertices $\widehat{Q}$.
If $R$ is a vertex lying on $\partial D$, consider a vertex $\widehat{Q}\in Int(D)$ that it is adjacent to $R$. 
Since $R\in E_{\widehat{Q}}$, it follows that $u_h(R)=m$ as well, showing that $u_h$ is constant.

For {\bf (B)}, if $\op{C}_m \subseteq \partial D$, then $u_h$ attains its  minimum only on $\partial D$.
Hence, if $P$ is an interior vertex, then
\begin{align*}
u_h(P) > m = \min u_h^b = \min \left((u_h^b)^+-(u_h^b)^- \right) \ge \min -(u_h^b)^- = -\max (u_h^b)^-,
\end{align*}
showing a strict inequality holds in~\eqref{eq:sldmppos}. Independently of the condition $\op{C}_m \subseteq \partial D$,
if $m>0$, then
\begin{align*}
u_h(P) \ge m >0 \ge \min -(u_h^b)^- = -\max (u_h^b)^-,
\end{align*}
which shows, again, that~\eqref{eq:sldmppos} holds  strictly. This leaves us with the case when $\op{C}_m \cap Int(D) \ne \varnothing$
and $m\le 0$. Let $Q\in\op{C}_m \cap Int(D)$. 
As in case {\bf (A)}, we will show that $u_h$ is constant on the set $\overline{E}_{Q}$ from the hypothesis, and the rest 
of the argument follows the same path as in the case {\bf (A)}. 
The focus is on $u_h|_{E_Q}$, to which we apply the conditions (B1) and (B2). By (B1) we have 
\begin{align}
\label{eq:m1ineq}
0\ge m = u_h(Q) \ge -\max (u_h^b)^- = m_1 = -(u_h^b)^-(Q_1), 
\end{align}
for some $Q_1\in \partial E_Q$. If $m_1=0$, then by~\eqref{eq:m1ineq} we have $m=m_1=0$; hence, Condition~(B2) implies
$u_h|_{E_Q}$ is constant. If $m_1<0$, then $(u_h^b)^-(Q_1) = -m_1 > 0$, which implies (by the definition of $(u_h^b)^-$)
that $(u_h^b)^-(Q_1) = -u_h^b(Q_1)$, showing  that  $u_h^b(Q_1) = m_1$. Cf.~\eqref{eq:m1ineq}
$$m = u_h(Q) \ge u_h(Q_1) = u_h^b(Q_1)= m_1.$$ 
Since $m$  is the global minimum of $u_h$, it follows that $m=m_1$. Now (B2) implies that $u_h$ is constant on $\overline{E}_Q$.\\
The proof for {\bf (R)} is identical to that of {\bf (A)}, except we assume from the beginning that $m\ge 0$.
\end{proof}
In practice we verify a stronger condition, in the sense that we may be able to cover $Int(D)$ with the interiors of 
sets on which sDMP-A, sDMP-B, or  sDMP-R, holds, as shown in the following result.
\begin{corollary}
\label{cor:DMP}
Let $(\op{S}^E_h)_{E\subseteq D}$ be a {consistent family of solution operators},
and the triangulation $\op{T}_h$ be so that the graph of the interior vertices {\color{black}is connected}, 
and that every boundary vertex is adjacent to an interior vertex. Let $X$ stand for any of the symbols
$A$, $B$, or $R$. Assume that 
there exists a finite set of discrete subdomains $(E_{i})_{i\in I}$ so that 
\begin{align}
\label{eq:corDMP} 
Int(D) = \cup_{i\in I} Int(E_{i})
\end{align}
and $\op{S}^{E_i}_h$ satisfies sDMP-X for all $i\in I$. Then $\op{S}^{D}_h$ satisfies sDMP-X.
\end{corollary}
\begin{proof}
This follows easily from Theorem~\ref{thm:DMP}, since for each interior vertex $Q$ there exists $i\in I$ so that
$Q\in Int(E_i)$, and $\op{S}^{E_i}_h$ satisfies sDMP-X. Conditions (A1)--(A2) (or (B1)--(B2), for $X=B$) are clearly verified,
since they are weaker than those of sDMP-A and sDMP-R (or SDMP-B).
\end{proof}

\subsection{Matrix form of sDMP-A and sDMP-B for linear elliptic equations}
\label{ssec:matrixform}
In this section we translate sDMP-A and sDMP-B in matrix form for the case of linear elliptic PDEs, 
namely when $c$ has the form~\eqref{eq:Creaclinear}. This form will facilitate the application of Corollary~\ref{cor:DMP}
to specific examples, as shown in Section~\ref{sec:linelliptic}. We have separate results for the case $\tilde{c}\equiv 0$, where sDMP-A is relevant,
and $\tilde{c}\ge 0$, where we look for sDMP-B to be satisfied, as in the continuous case.

For the purpose of this section we regard vectors as column  matrices. Given  $\mat{B}\in \R^{m\times n}$, we 
say that $\mat{B}$ is nonnegative and write $\mat{B}\ge \mat{0}$ if $\mat{B}_{ij}\ge 0$ for all $i,j$; we call 
$\mat{B}$ positive and write $\mat{B} > \mat{0}$ if $\mat{B}_{ij} > 0$ for all $i,j$. We also denote $\mat{B}\gneqq \mat{0}$
if $(\mat{B}\ge \mat{0}$ and $\mat{B}\ne \mat{0})$. Note that $\mat{B} > \mat{0}$ if and only if $\mat{B}\mat{x} > \mat{0}$ for every
vector that satisfies $\mat{x} \gneqq \mat{0}$. We write $\mat{B}\ge (>) \mat{C}$, if $\mat{B}-\mat{C} \ge (>) \mat{0}$.
If $\alpha\subseteq \{1,\dots, m\}$,  and $\beta\subseteq \{1,\dots, n\}$, let 
$\mat{B}_{\alpha \beta}  = (\mat{B}_{i j})_{i\in \alpha, j\in \beta}$; if $\mat{x}\in \R^n$, then 
$\mat{x}_{\beta}  = (\mat{x}_{j})_{j\in \beta} \in \R^{|\beta|}$, where $|\beta|$ denotes the cardinality of $\beta$.
We also define the column vector $\mat{1}=\lbrack 1,\dots,1\rbrack^T$.

We return to the matrix formulation of~\eqref{eq:weakelldisc} on the domain $D$ with $c$ as in~\eqref{eq:Creaclinear}.
We define the stiffness, mass, and reaction matrices $\mat{A}, \mat{M}, \mat{C}\in \R^{N\times N}$
associated with~\eqref{eq:weakelldisc} by 
\begin{equation}
\label{eq:stiffmassreac}
\mat{A}_{i j}=a_D(\varphi_j,\varphi_i),\ \  \mat{M}_{i j}=\linnprd{\varphi_j}{\varphi_i}_D,\ \ \mathrm{ and}\ \ 
\mat{C}_{i j}=\linnprd{\tilde{c}\varphi_j}{\varphi_i}_D,\ \ \mathrm{respectively}.
\end{equation}
If $\tilde{c}$ is constant, then
$\mat{C} = \tilde{c}\mat{M}$. Let $\mat{F}\in \R^N$ be given by $\mat{F}_i = \innprd{f}{\varphi_i}$. 
\begin{lemma} 
\label{lma:matformsdmpA}
Assume $\tilde{c}\equiv 0$, and let $\alpha$ and $\beta$ be the set of interior, respectively boundary nodes of the triangulation.
Then $\op{S}^D_h$ satisfies sDMP-A if and only if
\begin{equation}
\label{eq:matformsdmpA}
\mat{A}_{\alpha \alpha}^{-1} > \mat{0}\ \ \mathrm{and}\ \ \mat{A}_{\alpha \alpha}^{-1}\mat{A}_{\alpha \beta} <\mat{0}.
\end{equation}
\end{lemma}
\begin{proof}
We restrict our attention to the case when $f$ in~\eqref{eq:weakelldisc} is nonnegative.
The matrix form of~\eqref{eq:weakelldisc} with $c\equiv 0$ is
\begin{equation}
\label{eq:matforFEczero}
\mat{A}_{\alpha \alpha} \mat{u}_{\alpha} +  \mat{A}_{\alpha \beta} \mat{u}_{\beta}  = \mat{F}_{\alpha}.
\end{equation}
Note that $\mat{A}\mat{1} = \mat{0}$, which implies
\begin{equation}
\label{eq:stiffnessone}
\mat{A}_{\alpha \alpha} \mat{1}_{\alpha} +  \mat{A}_{\alpha \beta} \mat{1}_{\beta}  = \mat{0}_{\alpha}.
\end{equation}
Hence,
\begin{eqnarray}
\label{eq:1alpha}
\mat{1}_{\alpha} =   -\mat{A}_{\alpha \alpha}^{-1}\mat{A}_{\alpha \beta} \mat{1}_{\beta}.
\end{eqnarray}
Assume~\eqref{eq:matformsdmpA} holds. Let {\color{black}$u_h^b \in V^h(\partial D)$} be arbitrary, with vector representation
$\mat{u}_{\beta}\in \R^{|\beta|}$, and $u_h = \op{S}^D_h(f,u_h^b)$, with vector representation (in the nodal basis)
$\mat{u}$. Let $\underline{u} = \min u_h^b = \min_i \{\mat{u}_i\;:\;i\in \beta\}$.
Cf.~\eqref{eq:matforFEczero} and  \eqref{eq:1alpha} we have
\begin{eqnarray}
\nonumber
\mat{u}_{\alpha}-\underline{u}\mat{1}_{\alpha}& = & 
\mat{A}_{\alpha \alpha}^{-1}\left( \mat{F}_{\alpha} - \mat{A}_{\alpha \beta} \mat{u}_{\beta} \right)
+\underline{u} \mat{A}_{\alpha \alpha}^{-1}\mat{A}_{\alpha \beta} \mat{1}_{\beta}\\
\label{eq:matineq1}
&= & \mat{A}_{\alpha \alpha}^{-1}\mat{F}_{\alpha} - 
\mat{A}_{\alpha \alpha}^{-1}\mat{A}_{\alpha \beta} \left(\mat{u}_{\beta} 
-\underline{u}\mat{1}_{\beta}\right) \ge \mat{0},
\end{eqnarray}
where we used~\eqref{eq:matformsdmpA} together with $\mat{u}_{\beta} -\underline{u}\mat{1}_{\beta}\ge \mat{0}$ and  
$\mat{F}_{\alpha} \ge \mat{0}$. Therefore,
$$\min_D u_h \ge \underline{u} = \min_{\partial D}u_h^b.$$
Moreover, if $\mat{u}_{\beta} -\underline{u}\mat{1}_{\beta} \gneqq \mat{0}$ or $\mat{F}_{\alpha}\gneqq \mat{0}$, then  
$-\mat{A}_{\alpha \alpha}^{-1}\mat{A}_{\alpha \beta} <\mat{0}$ together with $\mat{A}_{\alpha \alpha}^{-1}> \mat{0}$
and~\eqref{eq:matineq1} imply 
that $\mat{u}_{\alpha}-\underline{u}\mat{1}_{\alpha} > \mat{0}$, showing that $u_h$ cannot attain the global minimum
$\underline{u}$ in the interior of $D$. So if $u_h$ attains the value $\underline{u}$ in the interior of $D$,
we have we must have $\mat{u}_{\beta} -\underline{u}\mat{1}_{\beta} = \mat{0}$ and $\mat{F}_{\alpha} = \mat{0}$. 
Now~\eqref{eq:matineq1} implies $\mat{u}_{\alpha}-\underline{u}\mat{1}_{\alpha} = \mat{0}$, proving that $u_h$ is constant.
This shows that $\op{S}^D_h$ satisfies sDMP-A.

To prove the reverse implication, assume $\op{S}^D_h$ satisfies sDMP-A. For an internal node $i\in \alpha$ define the source
$f^{(i)}$ so that {\color{black}$\innprd{f^{(i)}}{\varphi_j} = \delta_{ij}$}, i.e., $f^{(i)}$ is represented by the basis element
$\mat{e}_i$; let $u_h^b\equiv  0$. Due to sDMP-A,
the vector $\mat{u}^{(i)}$ representing the solution {\color{black}$u_h^{(i)} = \op{S}^D_h(f^{(i)}, 0)$} satisfies 
$$\min \mat{u}^{(i)} \ge \mat{0}.$$
In addition, if $\mat{u}^{(i)}$ has any zero entry, then $u_h^{(i)}$ has a global minimum in the interior of $D$, so by
sDMP-A, $u_h^{(i)}\equiv 0$, which, in turn, would imply $f^{(i)}\equiv 0$. 
This shows $\mat{u}^{(i)} > \mat{0}$, meaning, the column corresponding to $i$ in 
$\mat{A}_{\alpha \alpha}^{-1}$ is positive. Therefore, $\mat{A}_{\alpha \alpha}^{-1} > \mat{0}$. A similar 
argument, this time taking $f\equiv 0$ and boundary vectors of the form $\mat{u}_{\beta} = \mat{e}_j$ with $j\in \beta$
shows that $-\mat{A}_{\alpha \alpha}^{-1}\mat{A}_{\alpha \beta} > \mat{0}$.
\end{proof}
{\color{black}
We remark that the vector $\mat{u}^{(i)}$ in the proof of the reverse implication  is the vector representation of 
discrete Green's function with Dirac impulse forcing $f(x) = \delta(x-P_i)$. Hence, the matrix $\mat{A}_{\alpha \alpha}^{-1}$
represents the discrete Green's function, and the first of the conditions in~\eqref{eq:matformsdmpA} states that the 
discrete Green's function be positive. Oftentimes this is regarded as being equivalent to the DMP. However, we note that this form of the DMP only 
applies to zero--Dirichlet boundary conditions.}

\begin{lemma} 
\label{lma:matformsdmpB}
Assume $\tilde{c}\ge  0$, and
let $\alpha$ and $\beta$ be the set of interior, respectively boundary nodes of the triangulation.
Then $\op{S}^D_h$ satisfies sDMP-B if and only if
\begin{equation}
\label{eq:matformsdmpB}
(\mat{A}_{\alpha \alpha}+\mat{C}_{\alpha \alpha})^{-1} > \mat{0}\ \ \mathrm{and}\ \ 
(\mat{A}_{\alpha \alpha}+\mat{C}_{\alpha \alpha})^{-1}(\mat{A}_{\alpha \beta} + \mat{C}_{\alpha \beta})<\mat{0}.
\end{equation}
\end{lemma}
\begin{proof}
Again, we restrict our attention to the case when $f$ is nonnegative. 
The matrix form of~\eqref{eq:weakelldisc} with $\tilde{c} \ge  0$ is
\begin{equation}
\label{eq:matforFEcpos}
(\mat{A}_{\alpha \alpha} +\mat{C}_{\alpha \alpha})\mat{u}_{\alpha} +  
(\mat{A}_{\alpha \beta} + \mat{C}_{\alpha \beta})\mat{u}_{\beta}  = \mat{F}_{\alpha}.
\end{equation}
Using~\eqref{eq:stiffnessone}, we have 
\begin{eqnarray}
\label{eq:stiffCone}
(\mat{A}_{\alpha \alpha} + \mat{C}_{\alpha \alpha})\mat{1}_{\alpha} + (\mat{A}_{\alpha \beta} + \mat{C}_{\alpha \beta})\mat{1}_{\beta}
= \mat{C}_{\alpha \alpha} \mat{1}_{\alpha} + \mat{C}_{\alpha \beta}\mat{1}_{\beta}.
\end{eqnarray}
After multiplying~\eqref{eq:stiffCone} by a scalar $r$ and subtracting from~\eqref{eq:matforFEcpos} we obtain
\begin{eqnarray*}
(\mat{A}_{\alpha \alpha} +\mat{C}_{\alpha \alpha})(\mat{u}_{\alpha} - r \mat{1}_{\alpha})  +  
(\mat{A}_{\alpha \beta} + \mat{C}_{\alpha \beta})(\mat{u}_{\beta} - r \mat{1}_{\beta})   = 
\mat{F}_{\alpha} - r (\mat{C}_{\alpha \alpha} \mat{1}_{\alpha} + \mat{C}_{\alpha \beta}\mat{1}_{\beta}).
\end{eqnarray*}
Hence,
\begin{eqnarray}
\label{eq:stiffConeB}
\mat{u}_{\alpha} - r \mat{1}_{\alpha} &=& 
-\: (\mat{A}_{\alpha \alpha} +\mat{C}_{\alpha \alpha})^{-1}(\mat{A}_{\alpha \beta} + \mat{C}_{\alpha \beta})(\mat{u}_{\beta} - r \mat{1}_{\beta})  \\
\nonumber&&{+}\  (\mat{A}_{\alpha \alpha} +\mat{C}_{\alpha \alpha})^{-1}\mat{F}_{\alpha} \\
\nonumber&&-\ r(\mat{A}_{\alpha \alpha} +\mat{C}_{\alpha \alpha})^{-1} (\mat{C}_{\alpha \alpha} \mat{1}_{\alpha} + \mat{C}_{\alpha \beta}\mat{1}_{\beta}).
\end{eqnarray}
First we assume the conditions~\eqref{eq:matformsdmpB} hold. Then 
{\color{black}$$(\mat{A}_{\alpha \alpha} +\mat{C}_{\alpha \alpha})^{-1}\mat{F}_{\alpha}\ge 0.$$}
Let $\underline{u} = \min u_h^b = \min \mat{u}_\beta$. If $\underline{u} \ge 0$, then $\mat{u}_{\beta}^- = \mat{0}$. By taking $r=0$ 
in~\eqref{eq:stiffConeB} we get 
\begin{eqnarray*}
\mat{u}_{\alpha} \ge 
\overbrace{-\: (\mat{A}_{\alpha \alpha} +\mat{C}_{\alpha \alpha})^{-1}(\mat{A}_{\alpha \beta} + \mat{C}_{\alpha \beta})}^{> \mat{0}}
\mat{u}_{\beta} + \overbrace{(\mat{A}_{\alpha \alpha} +\mat{C}_{\alpha \alpha})^{-1}}^{>\mat{0}}\mat{F}_{\alpha}\ge \mat{0} ,
\end{eqnarray*}
which shows that $\min \mat{u}_{\alpha} \ge 0 = -\max \mat{u}_{\beta}^-$. If one of the coordinates of $\mat{u}_{\alpha}$ is 0, then
we must have $\mat{u}_{\beta}=\mat{0}$ and $\mat{F}_{\alpha} = \mat{0}$; from~\eqref{eq:matforFEcpos} it follows that 
$\mat{u}_{\alpha} =\mat{0}$.

If $\underline{u} < 0$, then we take $r=\underline{u}$ in~\eqref{eq:stiffConeB}, and we get
\begin{eqnarray*}
\mat{u}_{\alpha} - \underline{u} \mat{1}_{\alpha} &=& 
\overbrace{-\: (\mat{A}_{\alpha \alpha} +\mat{C}_{\alpha \alpha})^{-1}(\mat{A}_{\alpha \beta} + \mat{C}_{\alpha \beta})}^{> \mat{0}}
\overbrace{(\mat{u}_{\beta} - \underline{u} \mat{1}_{\beta})}^{\ge \mat{0}} \\
&&-\ \underline{u}\underbrace{(\mat{A}_{\alpha \alpha} +\mat{C}_{\alpha \alpha})^{-1} 
(\mat{C}_{\alpha \alpha} \mat{1}_{\alpha} + \mat{C}_{\alpha \beta}\mat{1}_{\beta})}_{\ge \mat{0}} 
+ \underbrace{(\mat{A}_{\alpha \alpha} +\mat{C}_{\alpha \alpha})^{-1}}_{>\mat{0}}\mat{F}_{\alpha}\ge \mat{0},
\end{eqnarray*}
where we used  $-\underline{u} > 0$ and  $\mat{C}\ge \mat{0}$. 
Therefore, $\mat{u}_{\alpha} - \underline{u} \mat{1}_{\alpha}\ge \mat{0}$, showing that
$$\min \mat{u}_{\alpha} \ge \underline{u} = -\max \mat{u}_{\beta}^-.$$
If $\mat{u}_{\alpha} - \underline{u} \mat{1}_{\alpha}$ has a coordinate that is zero, then we have
$$
(\mat{u}_{\beta} - \underline{u} \mat{1}_{\beta}) = \mat{0},\ \ \mat{F}_{\alpha} =  \mat{0},\ \ 
\underline{u} (\mat{C}_{\alpha \alpha} \mat{1}_{\alpha} + \mat{C}_{\alpha \beta}\mat{1}_{\beta}) =  \mat{0}.
$$
By applying~\eqref{eq:stiffConeB} with $r=\underline{u}$ we get $(\mat{u}_{\alpha} - \underline{u} \mat{1}_{\alpha}) = \mat{0}$,
showing that $\mat{u}$ is constant. Hence, we proved sDMP-B holds.

The proof of the reverse implication is similar to that in Lemma~\ref{lma:matformsdmpA}.
\end{proof}

\begin{remark}
\label{rem:secconditionDMP}
In practice, the second condition in~\eqref{eq:matformsdmpA} and~\eqref{eq:matformsdmpB} is verified
in the following way: for~\eqref{eq:matformsdmpA}, assuming $\mat{A}_{\alpha \alpha}^{-1} > \mat{0}$,
the second condition holds if $\mat{A}_{\alpha \alpha}^{-1}\mat{A}_{\alpha \beta} \le \mat{0}$
and for all $j\in \beta$ there exists $i\in \alpha$ so that $\mat{A}_{i j}<0$. This implies that every column $\mat{g}$
of $\mat{A}_{\alpha \beta}$ satisfies $\mat{g}\lneqq \mat{0}$, which renders $\mat{A}_{\alpha \alpha}^{-1}\mat{g} <\mat{0}$.
For~\eqref{eq:matformsdmpB}, we replace in the argument above
$\mat{A}_{\alpha \alpha}$ with $(\mat{A}_{\alpha \alpha} + \mat{C}_{\alpha \alpha})$
and $\mat{A}_{\alpha \beta}$ with $(\mat{A}_{\alpha \beta} + \mat{C}_{\alpha \beta})$.
\end{remark}

Due to the way it will be applied to linear elliptic equations, 
we reinterpret Corollary~\ref{cor:DMP} in this context as the following:
\begin{corollary}
\label{cor:sDMPmatform}Let $(\op{S}^E_h)_{E\subseteq D}$ be a {consistent family of solution operators}
associated with the discrete linear variational {\color{black}problem}~\eqref{eq:weakelldisc}.
We assume the triangulation $\op{T}_h$ is so that the graph of the interior vertices is connected, 
and that every boundary vertex is adjacent to an interior vertex. Furthermore, assume that 
there exists a finite set of discrete subdomains $(E_{i})_{i\in I}$ so that 
$Int(D) = \cup_{i\in I} Int(E_{i})$,
and denote by $\alpha_i$ the set of interior vertices of $E_i$ and by $\beta_i$ the set of boundary vertices of $E_i$.\\
{\bf ({\color{black}i})} If \  $\tilde{c}\equiv 0$ and for all $i\in I$ we have 
\begin{equation}
\label{eq:matformsdmpAthm}
\mat{A}_{\alpha_i \alpha_i}^{-1} > \mat{0}\ \ \mathrm{and}\ \ \mat{A}_{\alpha_i \alpha_i}^{-1}\mat{A}_{\alpha_i \beta_i} <\mat{0},
\end{equation}
then $\op{S}^{D}_h$ satisfies sDMP-A.\\
{\bf ({\color{black}ii})} If \  $\tilde{c}\ge 0$ and for all $i\in I$ we have 
\begin{equation}
\label{eq:matformsdmpBthm}
(\mat{A}_{\alpha_i \alpha_i}+\mat{C}_{\alpha_i \alpha_i})^{-1} > \mat{0}\ \ 
\mathrm{and}\ \ (\mat{A}_{\alpha_i \alpha_i}+\mat{C}_{\alpha_i \alpha_i})^{-1}
(\mat{A}_{\alpha_i \beta_i}+\mat{C}_{\alpha_i \beta_i}) <\mat{0},
\end{equation}
then $\op{S}^{D}_h$ satisfies sDMP-B.
\end{corollary}
Corollary~\ref{cor:sDMPmatform} offers a practical way to verify sDMP-A or sDMP-B {\bf globally} (on the entire domain $D$), 
by covering $D$ with patches where the corresponding DMP is satisfied, which can be verified by checking the conditions~\eqref{eq:matformsdmpAthm}
or~\eqref{eq:matformsdmpBthm} on each patch.

{\color{black}
We present a greedy algorithm that can be used in connection with Corollary~\ref{cor:sDMPmatform}. Let $\op{V}$ be the set of 
all the interior vertices. For a given vertex $P$ denote by $\op{V}_k(P)$ the set of all the vertices in $\op{V}$
that can be connected to $P$ via at most $k$ interior edges, and let
$$E^k_P = \bigcup\{ T\in \op{T}_h\ : \exists Q\in \op{V}_k(P)\ \mathrm{so\ that}\ Q\in \overline{T}\},$$ 
that is, the union of the stars around the vertices in $\op{V}_k(P)$.

\begin{algorithm}
{\color{black}
\begin{algorithmic}[1]
\State $\op{V}^{rem} :=  \op{V}$ 
\While{$\op{V}^{rem} \ne \varnothing$}
\State Select $P\in \op{V}^{rem}$
\State $k := 0$
\State sDMP := $False$
\Repeat{}
\State $k:= k+1$
\If{$\eqref{eq:matformsdmpAthm}$ holds on $E^k_P$}
\State sDMP := $True$
\EndIf
\State $\op{V}^{rem} := \op{V}^{rem} \setminus \op{V}_k(P)$
\Until{sDMP or $\op{V}^{rem}=\varnothing $}
\EndWhile
\State \Return sDMP
\end{algorithmic}
}
\caption{Greedy algorithm for verifying sDMP}
\label{alg:DMP}
\end{algorithm}

Algorithm~\ref{alg:DMP} seeks to cover the domain $D$ with patches of the form $E^k_{P_i}$, $k=1, 2, ...$,  for some interior vertices $P_i$, on
which  sDMP-A (or sDMP-B, respectively) is satisfied; the vertices $P_i$ are the ones selected at line 3. 
In order to switch from sDMP-A to sDMP-B, one needs to replace~\eqref{eq:matformsdmpAthm}
with~\eqref{eq:matformsdmpBthm} on line~8. Algorithm~\ref{alg:DMP} will return ''True'' if sDMP holds, and ''False'' otherwise.
The algorithm tracks the set of vertices that remain to be checked, which is denoted by $\op{V}^{rem}$, a set that will become empty in finitely many steps,
regardless of whether the variable ``sDMP'' becomes ``True'' or stays ``False''. Cf. Corollary~\ref{cor:sDMPmatform}, 
if  sDMP-A (or sDMP-B, respectively) fails on the entire domain $D$, there must be at least one interior vertex $P_{\mathrm{fail}}$ that does not lie in any patch 
$E^k_{P}$ on which sDMP-A (or sDMP-B, respectively) holds. Hence, the point $P_{\mathrm{fail}}$ will not be included in the 
``good'' patches of any other vertex selected 
at line 3. Moreover, $E^k_{P_{\mathrm{fail}}} = D$ for some $k>0$, showing that the variable ``sDMP''will remain ``False'' when returned. Thus, under the assumption of the
connectivity of the graph of the interior vertices, Algorithm~\ref{alg:DMP} always provides the correct answer on the validity of the sDMP for linear problems.

In terms of the complexity, the majority of the work goes into the test of line~8, where~\eqref{eq:matformsdmpAthm} (or ~\eqref{eq:matformsdmpBthm}, respectively)
is verified; that cost is $O(N_k(P)^3)$ where $N_k(P)$ is the cardinality of $\op{V}_k(P)$, due to the 
necessity of computing the inverse of a matrix. For each selected $P$, prior to completing the loop in lines 6--12, Algorithm~\ref{alg:DMP} will run through line~8 
for $1, 2,\dots,k_P$. Hence, the cost of running the loop is
$
O\left(\sum_{i=1}^{k_P} N_i(P)^3\right).
$
If $I_{\mathrm{DMP}}$ is the set of vertices selected at line 3, then the total cost is
\begin{equation}
\label{eq:cost1}
C_{Alg_1} = \sum_{P\in I_{\mathrm{DMP}}} O\left(\sum_{i=1}^{k_P} N_i(P)^3\right).
\end{equation}
Under additional assumptions, this leads to a simple, worst-case scenario, upper bound of the computational cost. 
Let $k_{\max}$ be the largest value of $k$ that appears on line~7 of Algorithm~\ref{alg:DMP}. 
Assuming quasi-uniformity of the mesh, we have $N_k(P) \approx O(k^d)$. Hence,
\begin{equation}
\label{eq:cost2}
C_{Alg_1} = N \sum_{P\in I_{\mathrm{DMP}}} O\left(\sum_{i=1}^{k_{\max}} i^{3d}\right) = O(N k_{\max}^{3d+1}).
\end{equation}
The upper bound in~\eqref{eq:cost2} does not take into account the balance between a large $k_{\max}$ and the size of $I_{\mathrm{DMP}}$: if $k_{\max}$
is relatively large there will be fewer calls to line~3 in the algorithm, which may lead to $|I_{\mathrm{DMP}}| \ll N$. 
In the extreme case when sDMP is not satisfied, then it might even happen that $|I_{\mathrm{DMP}}| = 1$, 
in case  Algorithm~\ref{alg:DMP} starts at $P_{\mathrm{fail}}$. Hence the cost is just $O(k_{\max}^{3d+1})$; since
$N = O(k^d_{\max})$, then the cost is $O(N^{3+\frac{1}{d}})$. The last number is higher than verifying directly the sDMP on the entire mesh, which has a cost of $O(N^3)$. 
However, if $k^3_{\max} \ll N$ then Algorithm~\ref{alg:DMP} is significantly more efficient than a direct verification. In particular, when 
mesh refinement leads to a small number of isolated irregularities, as shown in Section~\ref{ssec:nonmonotedge}, then the cost
is $O(N)$.

}

\subsection{Sufficient conditions for a weak discrete maximum principle}
\label{ssec:moreDMPs}
In this section we use a perturbation argument and sDMP-R property to prove wDMP-A for the linear elliptic equation
under weaker conditions than in Corollary~\ref{cor:sDMPmatform}. This is relevant because there are standard meshes 
where the second inequality in~\eqref{eq:matformsdmpAthm} does not hold in its strict form.


The question of well-posedness of~\eqref{eq:weakelldisc} under these conditions will certainly arise and {\color{black}be}
handled properly, but the current focus is on the relevant DMP. 
\begin{lemma} 
\label{lma:matformsdmpRB}
Consider the context and hypotheses of Lemma~\textnormal{\ref{lma:matformsdmpB}}, except we now assume $\tilde{c}\le  0$. 
If~\eqref{eq:matformsdmpB} holds, then $\op{S}^D_h$ is well-posed and satisfies sDMP-R.
\end{lemma}
The proof follows closely that of Lemma~\ref{lma:matformsdmpB}.
\begin{proof}
Assume $f$ is nonnegative. Since the matrix form of~\eqref{eq:weakelldisc}  is given, as in Lemma~\ref{lma:matformsdmpB}, 
by~\eqref{eq:matforFEcpos}, the first condition in~\eqref{eq:matformsdmpB} ensures that the linear equation~\eqref{eq:matforFEcpos}
has a unique solution;
hence, the problem is well-posed.
Let $\underline{u} = \min u_h^b = \min \mat{u}_\beta$. We assume $\underline{u} \ge 0$. By taking $r=\underline{u}$ 
in~\eqref{eq:stiffConeB} we get 
\begin{eqnarray*}
\mat{u}_{\alpha} - \underline{u} \mat{1}_{\alpha} &=& 
\overbrace{-\: (\mat{A}_{\alpha \alpha} +\mat{C}_{\alpha \alpha})^{-1}(\mat{A}_{\alpha \beta} + \mat{C}_{\alpha \beta})}^{> \mat{0}}
\overbrace{(\mat{u}_{\beta} - \underline{u} \mat{1}_{\beta})}^{\ge \mat{0}} \\
&&-\ \underline{u}\underbrace{(\mat{A}_{\alpha \alpha} +\mat{C}_{\alpha \alpha})^{-1} 
(\mat{C}_{\alpha \alpha} \mat{1}_{\alpha} + \mat{C}_{\alpha \beta}\mat{1}_{\beta})}_{\le \mat{0}} 
+ \underbrace{(\mat{A}_{\alpha \alpha} +\mat{C}_{\alpha \alpha})^{-1}}_{>\mat{0}}\mat{F}_{\alpha}\ge \mat{0},
\end{eqnarray*}
where we used  $\underline{u} \ge 0$ and  $\mat{C}\le \mat{0}$. 
Therefore, $\mat{u}_{\alpha} - \underline{u} \mat{1}_{\alpha}\ge \mat{0}$; thus, $\min \mat{u}_{\alpha} \ge \underline{u}$.
If $\mat{u}_{\alpha} - \underline{u} \mat{1}_{\alpha}$ has a coordinate that is zero, then we have
$$
(\mat{u}_{\beta} - \underline{u} \mat{1}_{\beta}) = \mat{0},\ \ \mat{F}_{\alpha} =  \mat{0},\ \ 
\underline{u} (\mat{C}_{\alpha \alpha} \mat{1}_{\alpha} + \mat{C}_{\alpha \beta}\mat{1}_{\beta}) =  \mat{0}.
$$
By applying~\eqref{eq:stiffConeB} with $r=\underline{u}$ we get $(\mat{u}_{\alpha} - \underline{u} \mat{1}_{\alpha}) = \mat{0}$,
showing that $\mat{u}$ is constant. Hence, we proved that $\op{S}^D_h$ satisfies sDMP-R.
\end{proof}

\begin{theorem}[wDMP-A]
\label{thm:wdmpa}
Consider the common hypotheses and notation from Corollary~\ref{cor:sDMPmatform}, and assume $\tilde{c}\equiv 0$.
If for all $i\in I$ we have 
\begin{equation}
\label{eq:matformWdmpAthm}
\mat{A}_{\alpha_i \alpha_i}^{-1} > \mat{0}\ \ \ \mathrm{and}\ \ \ \mat{A}_{\alpha_i \beta_i} \le \mat{0},
\end{equation}
then $\op{S}^{D}_h$ satisfies wDMP-A.
\end{theorem}
Note the similarity with the statement from Corollary~\ref{cor:sDMPmatform}{\color{black}{\bf (i)}}; here we 
relaxed the  second condition in~\eqref{eq:matformsdmpAthm}, and the conclusion is a weaker form of the DMP, namely wDMP-A.
\begin{proof}
We begin by defining a sequence of families of consistent discrete  solution operators.
This family would, in principle, be the finite element discretization of~\eqref{eq:slcontdef}-\eqref{eq:slcontBC} with 
${c}(x,u) = -\varepsilon u$ and $\varepsilon>0$; however, neither the continuous, nor the discrete versions
are  guaranteed to be well-posed for a sufficiently rich family of  $\varepsilon>0$.
In order to circumvent this problem, we restrict our attention to discrete solution operators, 
and only to  a certain sequence $(\varepsilon_n)_{n\in \N}$ that converges to 0.

To define our solution operators first we note that the triangulation $\op{T}_h$ is fixed.
For {\bf every } discrete subdomain $E\subseteq D$ denote its sets of interior and boundary  vertices 
by $\alpha=\alpha(E)$ and $\beta=\beta(E)$, respectively. 
Define the number
\begin{eqnarray}
\label{eq:setsigma}
\delta = \min\bigcup_{E\subseteq D}\sigma(\mat{M}_{\alpha(E) \alpha(E)}^{-1}\mat{A}_{\alpha(E) \alpha(E)}),
\end{eqnarray}
with the union being taken over all the discrete subdomains, and $\sigma(\mat{B})$ denoting the spectrum of a 
matrix $\mat{B}$.
The expression in~\eqref{eq:setsigma} is well defined because the set on the right-hand side above is finite.
Moreover, since both 
$\mat{M}_{\alpha(E) \alpha(E)}$ and $\mat{A}_{\alpha(E) \alpha(E)}$ are symmetric positive definite matrices,
the set in~\eqref{eq:setsigma} lies in $(0,\infty)$, implying that $\delta>0$.
{\color{black}Let $\varepsilon_n$ be a sequence so that $0<\varepsilon_n<\delta$ and $\lim_{n\to\infty }\varepsilon_n  = 0$. Define }
\begin{eqnarray*}
\mat{B}^{(n)}_i = \mat{A}_{\alpha_i \alpha_i} - \varepsilon_n \mat{M}_{\alpha_i \alpha_i}.
\end{eqnarray*}
Since $\varepsilon_n\notin \sigma(\mat{M}_{\alpha(E) \alpha(E)}^{-1}\mat{A}_{\alpha(E) \alpha(E)})$, it follows that
$\mat{B}^{(n)}_i$ is invertible.
Due to the continuity of matrix inversion we have the following: for all $i\in I$
\begin{eqnarray*}
\lim_{n\to \infty} (\mat{B}^{(n)}_i)^{-1} = \mat{A}^{-1}_{\alpha_i \alpha_i} > \mat{0}.
\end{eqnarray*}
Hence, there exists $n_0\in \N$ so that for 
\begin{equation}
\label{eq:Bngzero}
\forall n\ge n_0,\ \ \forall i\in I,\ \ (\mat{B}^{(n)}_i)^{-1} > \mat{0}.
\end{equation}
It is assumed that  $\forall k\in \beta_i$, $\exists j\in \alpha_i$ so that
the boundary  vertex $P_k$ (of $E_i$) is connected to the interior vertex $P_j$; this translates
into $(\mat{M}_{\alpha_i \beta_i})_{jk}>0$, showing that every column of $\mat{M}_{\alpha_i \beta_i}$
has at least one positive entry.
Since $\mat{M}_{\alpha_i \beta_i} \ge \mat{0}$,~\eqref{eq:Bngzero} implies that 
\begin{eqnarray}
\label{eq:BnMnlzero1}
\forall n\ge n_0,\ \ \forall i\in I,\ \ 
(\mat{A}_{\alpha_i \alpha_i} - \varepsilon_n \mat{M}_{\alpha_i \alpha_i})^{-1} \mat{M}_{\alpha_i \beta_i} > \mat{0}.
\end{eqnarray}
Therfore, it follows from~\eqref{eq:BnMnlzero1},~\eqref{eq:matformWdmpAthm}, and $\varepsilon_n>0$ that
\begin{eqnarray}
\label{eq:BnMnlzero2}
\ \ \ \forall n\ge n_0,\ \  \forall i\in I,\ \ (\mat{A}_{\alpha_i \alpha_i} - \varepsilon_n \mat{M}_{\alpha_i \alpha_i})^{-1} 
(\mat{A}_{\alpha_i \beta_i}-\varepsilon_n\mat{M}_{\alpha_i \beta_i}) < \mat{0}.
\end{eqnarray}

Now we construct the family of consistent solution operators, namely we consider the following sequence of 
{\bf discrete} variational {\color{black}problems}: given a discrete subdomain $E\subseteq D$, 
find $u_h\in V^h(\overline{E})$ of the form $u_h = u_h^0+ u_h^b$
with $u_h^0\in V_0^h(\overline{E})$, so that 
\begin{eqnarray}
\label{eq:weakelldiscnegeps}
a_E(u_h^0+u_h^b,v) -\varepsilon_n\linnprd{u_h^0+u_h^b}{v}_E=\innprd{f}{v}_E,\ \ \forall v\in V_0^h(\overline{E}).
\end{eqnarray}
If $\alpha=\alpha(E)$ and $\beta=\beta(E)$, then~\eqref{eq:weakelldiscnegeps} is formulated in matrix form as
\begin{eqnarray}
\label{eq:weakelldiscnegepsmat}
(\mat{A}_{\alpha \alpha}-\varepsilon_n\mat{M}_{\alpha \alpha})\mat{u}_{\alpha} + 
(\mat{A}_{\alpha \beta} -\varepsilon_n\mat{M}_{\alpha \beta})\mat{u}_{\beta}  = \mat{F}_{\alpha}.
\end{eqnarray}
The choice of $\varepsilon_n$ ensures that~\eqref{eq:weakelldiscnegeps} is well posed (uniquely solvable), hence
this gives rise to a family of solution operators $\op{S}^E_{h,n}: (V_0^h(\overline{E}))^*\times V^h(\partial E) \to V^h(E)$,
so that $\op{S}^E_{h,n}(f,u_{\beta}) = u = u_{\alpha}+u_{\beta}$, with $\mat{u}_{\alpha}$ and $\mat{u}_{\beta}$ 
satisfying~\eqref{eq:weakelldiscnegepsmat}
being the vector representations of ${u}_{\alpha}$ and ${u}_{\beta}$, respectively.
Hence, Lemma~\ref{lma:matformsdmpRB} and~\eqref{eq:BnMnlzero2} imply that for all $i\in I$ and $n\ge n_0$, 
$\op{S}^{E_i}_{h,n}$ satisfies sDMP-R.
Due to the variational formulation~\eqref{eq:weakelldiscnegeps}, 
we can use an argument similar to that in Lemma~\ref{lma:restriction} to show that for all $n\ge n_0$, the family
$(\op{S}^E_{h,n})_{E\subseteq D}$ is consistent, as in Definition~\ref{def:soloperators}. Theorem~\ref{thm:DMP}{\bf (R)} now applies to show that
for all $n\ge n_0$, $\op{S}^{D}_{h,n}$ satisfies sDMP-R.

For the final step let 
$f \in (V_0^h(\overline{D}))^*$ be nonnegative and $u_h=\op{S}^{D}_{h}(f,u^b)$ 
be the solution of~\eqref{eq:weakelldisc} with $c\equiv 0$. 
Denote
\begin{eqnarray*}
\underline{u} = \min_{\overline{D}} u_h,\ \ \ 
v_h = u_h-\underline{u}+1,\ \ \ \mathrm{and}\ \  v_h^b = u_h^b-\underline{u}+1.
\end{eqnarray*} 
Since $c\equiv 0$, we have $v_h=\op{S}^{D}_{h}(f,v_h^b)$, because the pair
$\mat{v}_{\alpha}, \mat{v}_{\beta}$ given by 
$$
\mat{v}_{\alpha} = \mat{u}_{\alpha}-(\underline{u}-1)\mat{1}_{\alpha},\ \ \ 
\mat{v}_{\beta} = \mat{u}_{\beta}-(\underline{u}-1)\mat{1}_{\beta},
$$
where $\alpha=\alpha(D)$, $\beta=\beta(D)$
satisfy~\eqref{eq:matforFEczero}.
This way we ensure $v_h \ge 1$ (the number 1 is not special, all that matters is that $1>0$).
Define $v_{h,n}=\op{S}^{D}_{h,n}(f,v_h^b)$. Using~\eqref{eq:matforFEczero} 
and~\eqref{eq:weakelldiscnegepsmat}, we get 
$$\lim_{n\to\infty}v_{h,n} =  v_h \ge 1\ \ \mathrm{in}\ \ V^h(D)\ \ \mathrm{(pointwise\ or\ any\ norm)},$$ 
showing that $\exists  n_1\ge n_0$ so that for $n\ge n_1$ we have $v_{h,n}>0$. Since $\op{S}^{D}_{h,n}$ satisfies
sDMP-R, we get
{\color{black}
\begin{equation}
\label{eq:sdmprvn}
\min_{\overline{D}} v_{h,n} = \min_{\partial{D}} v_{h,n} = \min_{\partial{D}} v^b_{h},\ \ \forall n\ge n_1.
\end{equation}
}
After passing to the limit in~\eqref{eq:sdmprvn} we get
\begin{equation}
\label{eq:sdmpru}
\min_{\overline{D}} u_{h} = \min_{\overline{D}} v_{h} + \underbar{u}-1 = \min_{\partial{D}} v^b_{h} + \underbar{u}-1 =
\min_{\partial{D}} u^b_{h},
\end{equation}
showing that $\op{S}^{D}_{h}$ satisfies wDMP-A.
\end{proof}

\section{Application to linear elliptic equations}
\label{sec:linelliptic}
We continue to restrict our attention to ${\mathcal P}_1$ finite elements.
\color{black}
In this section we apply the results from Section~\ref{sec:mainres} to the linear version of~\eqref{eq:weakelldisc},
namely the case where $c$ takes the form~\eqref{eq:Creaclinear}. The focus is on
the connection to classical results, the angle condition, and the violation of the strong DMP due to mesh properties at the boundary.



\subsection{Connection to the classical results}
\label{ssec:classresult}
In this section we revisit classical results for meshes where the classical angle condition~\eqref{eq:anglecond} (or its weaker version) is satisfied. 
This is related to local DMPs holding on the smallest meaningful discrete units, namely the union of all triangles that contain a vertex. 

\begin{definition}
\label{def:connected}
We call $E\subseteq D$ a \emph{connected discrete subdomain} if the graph of the nodes that are interior  to $E$ 
is connected (using only interior edges), and every vertex on $\partial E$ is connected to a  vertex interior to $E$.
\end{definition}
The following is a variant of a classical result,  which we prove here using the connectivity  technique introduced in
Section~\ref{sec:mainres}.
\begin{theorem}
\label{thm:classicalDMP}
Assume $\tilde{c}\equiv 0$, and that the stiffness matrix $\mat{A}$ in~\eqref{eq:weakelldisc} satisfies
$\mat{A}_{ij}<0$ whenever the vertices $P_i$ and $P_j$ are adjacent with at least one of them lying in the interior 
(it does not need to hold if $P_i,P_j\in \partial D$). Then the discrete solution operator 
$\op{S}^E_h$ satisfies sDMP-A for every connected discrete subdomain $E\subseteq D$.
\end{theorem}
Note that this is the closest statement to the continuous case, where the maximum principle is satisfied on every
qualifying subdomain.
\begin{proof}
Let $E\subseteq D$ a \emph{connected discrete subdomain}. For every vertex $P_i$ interior to $E$ 
we consider the set $E_i = \bigcup \{T\in\op{T}_h\  :\ P_i\in T\}$, which we call the \emph{star} around $P_i$.
The boundary vertices of $E_i$ are the adjacent vertices of $P_i$, which is the only interior node of $E_i$. 
In the terminology of Lemma~\ref{lma:matformsdmpA}, $\alpha = \{i\}$ and {\color{black}$\beta=\{j\ :\ \mat{A}_{ij}<0\}$}, and the 
conditions~\eqref{eq:matformsdmpA} are satisfied in a trivial manner. Hence, $\op{S}^{E_i}_h$ satisfies sDMP-A.
Corollary~\ref{cor:DMP} implies that $\op{S}^{E}_h$ satisfies sDMP-A.
\end{proof}

\subsection{The angle condition}
\label{sec:angleclt}

For $\op{P}_1$ elements, the condition $\mat{A}_{ij}<0$ for adjacent vertices $P_i$ 
and $P_j$ is known as the ``angle condition,'' and it is
\color{black}
the backbone of the majority of results regarding DMPs, not just for $\op{P}_1$ elements, but also for 
$\op{Q}_1$~\cite{MR738731, MR2392463, korotov2010comparison}, $\op{P}_2$~\cite{MR632125}, and higher 
elements~\cite{MR2607805, MR2520861}.
The root of the name lies in the formula for the entries in the stiffness matrix for the Laplacian 
($a_{ji} = \delta_{ij}$ in~\eqref{eq:slcontdef}); 
If $P_i$ and $P_j$ are adjacent vertices, denote by $\op{T}_1$ and $\op{T}_2$ the triangles
bordered by the edge $e=\overline{P_1 P_2}$, and let $\theta_1, \theta_2$ be the angles opposite the edge 
$e$ in each of $\op{T}_1$ and $\op{T}_2$, respectively (see Fig.~\ref{fig:triangle_formula}, left). 
Cf. Lemma~A.1. in~\cite{MR2085400}, 
\begin{equation}
\label{eq:formula2triangle}
\mat{A}_{ij} = a_D(\varphi_j,\varphi_i) = {\color{black}\int_{\op{T}_1 \cup \op{T}_2} \nabla \varphi_i \cdot \nabla \varphi_j} = 
-\frac{\sin (\theta_1+\theta_2)}{2 \sin \theta_1 \sin \theta_2}.
\end{equation}
Thus we arrive at the following {\bf angle condition}:
\begin{equation}
\label{eq:anglecond}
\mat{A}_{ij} < 0\ \ \ \ \mathrm{iff}\ \ \ \  \theta_1+\theta_2 < \pi.
\end{equation}
We should also note the weaker condition: $\mat{A}_{ij} \le 0$ iff $\theta_1+\theta_2 \le \pi$. Oftentimes the angle condition
is remembered as ``all angles must be acute'', a hypothesis that certainly implies~\eqref{eq:anglecond}.
\begin{figure}[!h]
\begin{center}
        \includegraphics[width=5.0in]{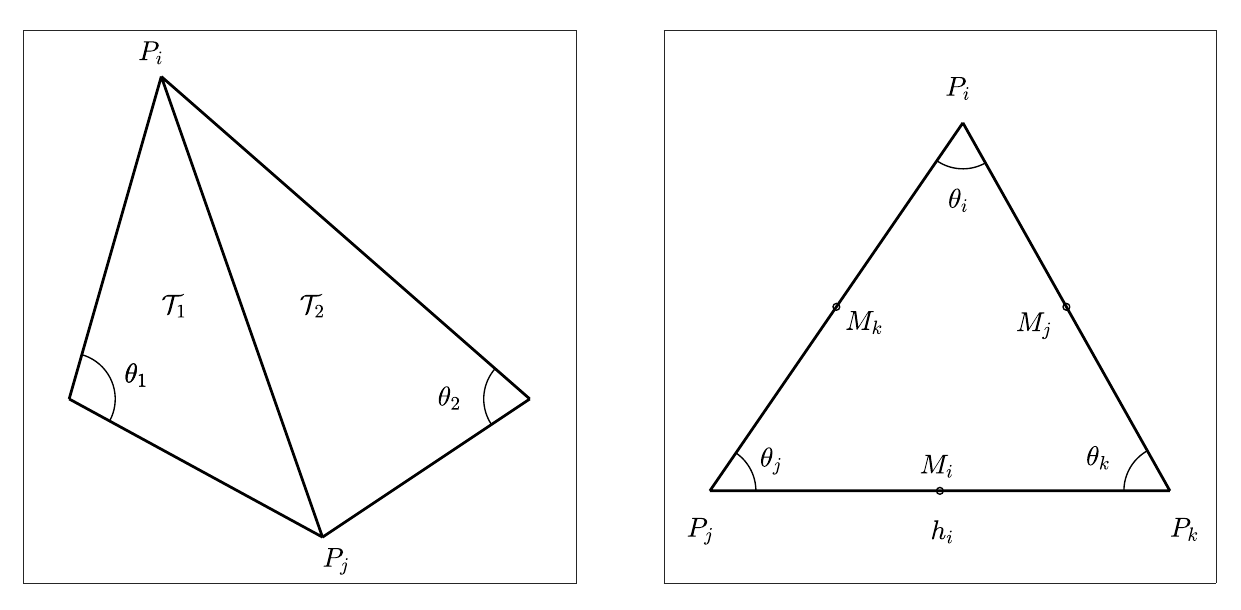}
\end{center}
\caption{Triangle elements entering the formulas for stiffness and mass matrices.}
\label{fig:triangle_formula}
\end{figure}
We also recall from (see~\cite{MR2085400}) the formulas for the diagonal entries of the stiffness matrix $\mat{A}$.
Consider a mesh triangle $\op{T}$ with vertices $P_i, P_j, P_k$ and  angles $\theta_i, \theta_j, \theta_k$ (refer to Fig.~\ref{fig:triangle_formula}, right),
and let $\varphi_i$ be the basis function associated with $P_i$. The contribution to $\mat{A}_{ii}$ from $\op{T}$ is given by
\begin{equation}
\label{eq:formula1triangle}
\int_{\op{T}}|\nabla \varphi_i|^2 = \frac{\sin \theta_i}{2 \sin \theta_j \sin \theta_k}.
\end{equation}
Fo the mass matrix we use the the fact that the cubature rule 
\begin{eqnarray*}
Q(f) = \frac{\mu(\op{T})}{3} \sum_{\ell\in\{i,j,k\}} f(M_{\ell})
\end{eqnarray*} 
is exact for quadratics, where $M_i, M_j, M_k$ are the edge midpoints (see Fig.~\ref{fig:triangle_formula}, right), {\color{black}
and $\mu(\op{T})$ denotes the area of $\op{T}$}.
The contributions of $\op{T}$ to the $i^{\mathrm{th}}$ row of the mass matrix are
\begin{eqnarray}
\label{eq:formula3triangle}
\int_{\op{T}}\varphi_i^2 &=& \frac{\mu(\op{T})}{3} \sum_{\ell\in\{i,j,k\}}\varphi^2_i(M_{\ell})  = \frac{\mu(\op{T})}{6} = 
\frac{h_i^2}{12(\cot \theta_j + \cot \theta_k)},\\
\label{eq:formula4triangle}
\int_{\op{T}}\varphi_i \varphi_j &=& \frac{\mu(\op{T})}{3} \sum_{\ell\in\{i,j,k\}}
\varphi_i(M_{\ell})\varphi_j(M_{\ell}) = \frac{\mu(\op{T})}{12} = \frac{h_k^2}{24(\cot \theta_i + \cot \theta_j)},
\end{eqnarray}
where we use the area formula (and permutations)
$$
\mu(\op{T}) = \frac{h_i^2}{2(\cot \theta_j + \cot \theta_k)}.
$$

Note that there are standard meshes of interest containing edges $\overline{P_i P_j}$ for which $\mat{A}_{ij} = 0$, where wDMP-A is known to hold.
In Example~\ref{ex:threelinesmesh} we analyze such a mesh which can be regarded as a limiting case of a mesh that satisfies sDMP-A (and sDMP-B).

\subsection{Defects related to boundary mesh properties}
\label{sec:nonmonotedgebdry}

We first consider defects associated to mesh properties at the boundary of the domain.
To begin with, consider a mesh with a triangle $T$ having two edges on the boundary,
which meet at the boundary vertex $P$.
Let $V_0^h(\overline{D})$ denote all piecewise linear functions that vanish on $\partial D$.
A function in $\psi\in V^h(\overline{D})$ that satisfies
$$
a_D(\psi,v)+(c(\cdot,\psi),v)_D=0
$$
for all $v\in V_0^h(\overline{D})$ is often referred to as a discrete harmonic 
function \cite{heilbronn1949discrete,MR551291}.

\begin{lemma}\label{lem:surpriz}
Let $\phi$ denote the Lagrange basis function that is 1 at $P$ and zero at all other vertices.
Consider the variational problem defined in \eqref{eq:weakellcont}.
Then 
$$
a_D(\phi,v)+(c(\cdot,\phi),v)_D=0
$$
for all $v\in V_0^h(\overline{D})$.
Thus $\phi$ is discrete harmonic but supported in $T$, and
$u_0$ defined by \eqref{eq:weakelldisc} is {\color{black}identically} zero when $f\equiv 0$.
\end{lemma}

\begin{proof}
For $v\in V_0^h(\overline{D})$, the supports of $v$ and $\phi$ intersect in a set of measure zero,
the third edge of $T$.
\end{proof}

\begin{example}
\label{ex:threelinesmesh}
Let the two-dimensional vectors be $\mat{v}_{\theta}=[\cos \theta, \sin \theta]^T$,
so that $\mat{v}_0=[1, 0]^T$.
For $\pi/2 \le \theta< \pi$, define  the rhombus
\color{black}
$D(\theta) = \{ s\: \mat{v}_0 + t\: \mat{v}_{\theta}\ :\ 0 < s, t\ <1\}$.
We consider the Poisson equation on $D(\theta)$ 
($a_{ji} = \delta_{ij}$ and $c \equiv 0$ in~\eqref{eq:slcontdef}).
To obtain a triangular mesh we  partition $D(\theta)$ in $n\times n$ identical rhombuses, which we further divide into triangles
along their short diagonal, as pictured in Fig.~\ref{fig:structuredmesh} -- right. Consider a standard lexicographic numbering of the 
vertices, numbered from $1$ to  $N=(n+1)\times (n+1)$. The case $ \theta = \pi/2$ leads to 
$D(\theta) = [0,1]\times [0,1]$ with the classical  three-line mesh discretization, as shown in Fig.~\ref{fig:structuredmesh} -- left. \\
{\bf Case 1: $\theta >\pi/2$} (Fig.~\ref{fig:structuredmesh} -- right). We first note that the boundary values imposed at the vertices
$P_{n+1} = (1,0)$ and $P_k = (\cos \theta,\sin\theta)$, with $k = 1+ n (n+1)$, do not influence the solution in the interior; that is, 
 $u^0_h$ is the same regardless of the values $u^b_h(P_{n+1})$ and $u^b_h(P_k)$. This shows $S^D_h$ does not satisfy sDMP-A, since 
attaining a minimum for $u_h^0$ in the interior does not imply $u_h=u^0_h+u^b_h$ is constant. However, if 
${\tilde{D}}$ denotes the domain $D$ from which we remove the  triangles  containing the vertices $P_{n+1}$ and $P_k$ 
(shaded in Fig.~\ref{fig:structuredmesh} -- right), then $S^{\tilde{D}}_h$ satisfies sDMP-A, since all the angles are acute, and every
boundary node is connected to an internal node. {\color{black}Furthermore}, if we now add back the triangles to the domain, then for any
nonnegative $f$ and discrete boundary function $u^b_h\in V^h(\partial D)$,  the interior part of the solution $u^0_h$ satisfies
\begin{eqnarray*}
\label{eq:wDMPrhombus}
\min_{P\in Int(D)} u^0_h(P) & =& \min_{P\in Int(\tilde{D})} u^0_h(P) \ge  \min_{P\in \partial\tilde{D}} u^b_h(P)\\
& \ge& \min 
\{\min_{P\in \partial\tilde{D}} u^b_h(P), u^b_h(P_{n+1}), u^b_h(P_k)\} = \min_{P\in \partial {D}} u^b_h(P).
\end{eqnarray*}
This shows that $S^{{D}}_h$ satisfies wDMP-A. We will show in Section~\ref{sec:semilinear} that $S^{\tilde{D}}_h$ also 
satisfies sDMP-B when $c$ is non-zero.
\\
{\bf Case 2: $\theta = \pi/2$} (Fig.~\ref{fig:structuredmesh} -- left).
In this case~\eqref{eq:formula2triangle} implies that $\mat{A}_{ij}=0$ for every edge $\overline{P_i P_j}$ parallel to the
line $y=x$.
This shows that the boundary values at all the four corners of $D$ do not influence the computed (interior) solution $u^0_h$,
so {\color{black}$S^{{D}}_h$ does not satisfy sDMP-A}. In this case even $S^{\tilde{D}}_h$ does not satisfy sDMP-A due to the values
at $P_1$ and $P_N$ not influencing $u^0_h$. This does not change by removing the triangles containing $P_1$ and $P_N$, 
as doing so will give rise to new ``disconnected'' boundary vertices. However, Theorem~\ref{thm:wdmpa} implies that
$S^{\tilde{D}}_h$ satisfies wDMP-A, as we choose for each interior
point $P_i$ the star $E_i$ as in the proof of Theorem~\ref{thm:classicalDMP}. 
Note that the hypotheses of Theorem~\ref{thm:wdmpa} are satisfied, because we have
\begin{equation}
\label{eq:nonposent}
{\color{black}\mat{A}_{ij}\le 0},\ \ \mathrm{for\  all}\ \  i\ne j.
\end{equation} 
The same calculation as in Case 1 shows that also $S^{{D}}_h$ satisfies wDMP-A.
\begin{figure}[!h]
\begin{center}
        \includegraphics[width=4.5in]{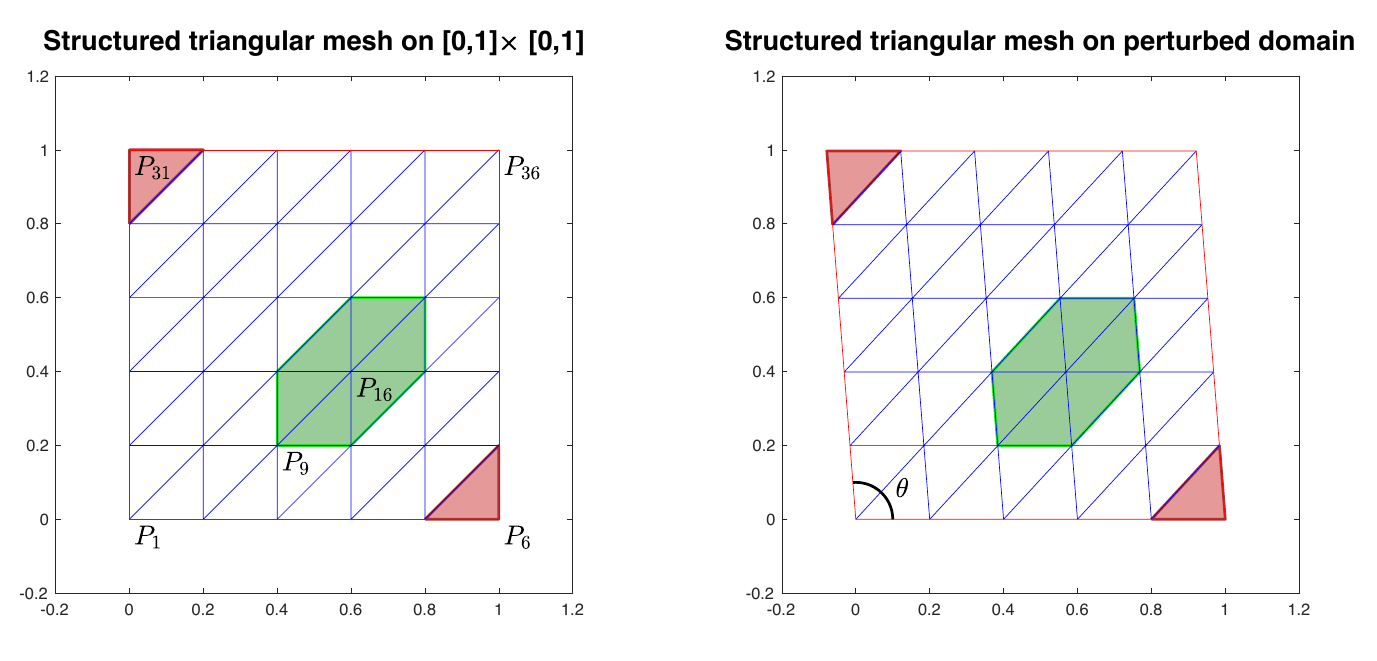}
\end{center}
\caption{\emph{Left:} The Poisson solution operator $S^D_h$ satisfies wDMP-A
on the classical three-line mesh on $D = [0,1]\times [0,1]$ (here $n=5$), but does not satisfy sDMP-A; 
entries in the stiffness matrix corresponding to
edges like $\overline{P_9, P_{16}}$ are zero. \emph{Right:} If $\tilde{D}$ is obtained by removing from $D$ the triangles containing
$P_6$ and $P_{31}$, then sDMP-A holds on $\tilde{D}$, and wDMP-A holds on ${D}$.}
\label{fig:structuredmesh}
\end{figure}
\end{example}
\color{black}



\section{
Defects related to interior mesh properties}
\label{ssec:nonmonotedge}

\color{black}
In this section we focus on the Poisson equation with a constant reaction rate $\tilde{c}$; thus, the reaction
matrix has the form $\mat{C} = \tilde{c}\: \mat{M}$. We present 
a set of examples where the angle condition~\eqref{eq:anglecond} (which still refers to the stiffness matrix)
fails for certain edges, and yet we can prove that the sDMP-A and sDMP-B hold, under appropriate assumptions.

\subsection{A class of meshes with defects}
\label{ssec:defects}
For the purpose of this section we call edges where~\eqref{eq:anglecond} fails ``defects''.
The domain is related to that in Example~\ref{ex:threelinesmesh}
via a 90 degree rotation, 
\color{black}
namely we consider a rhombus $D(\theta)$ centered at 
the origin, with the long diagonal lying on the $x$-axis, and $0<\theta < \pi/2$ representing the acute
angle of the rhombus.
We partition $D(\theta)$ into $n\times n$ identical rhombuses, which we further divide 
either along the short or the long diagonal, according to a set of rules described below. First,
all the rhombuses that contain a boundary edge are divided along their short diagonal, as
in Example~\ref{ex:threelinesmesh}. Second, we eliminate the two corner triangles that cross the $x$-axis
(see Figures~\ref{fig:G1mesh}--\ref{fig:G2meshdefect}), and we denote by $\tilde{D}(\theta)$ the remaining
domain; note that $\tilde{D}(\theta)$ changes as the mesh is refined.
Third, the defects are selected from a finite library described below, and they need to be separated by triangles 
in which all the edges satisfy the angle condition~\eqref{eq:anglecond}. Our presentation does not
aim to exhaust all the possible patterns; instead, we want to showcase the usage  our results on some nontrivial examples.

\begin{figure}[!h]
\begin{center}
        \includegraphics[width=3in]{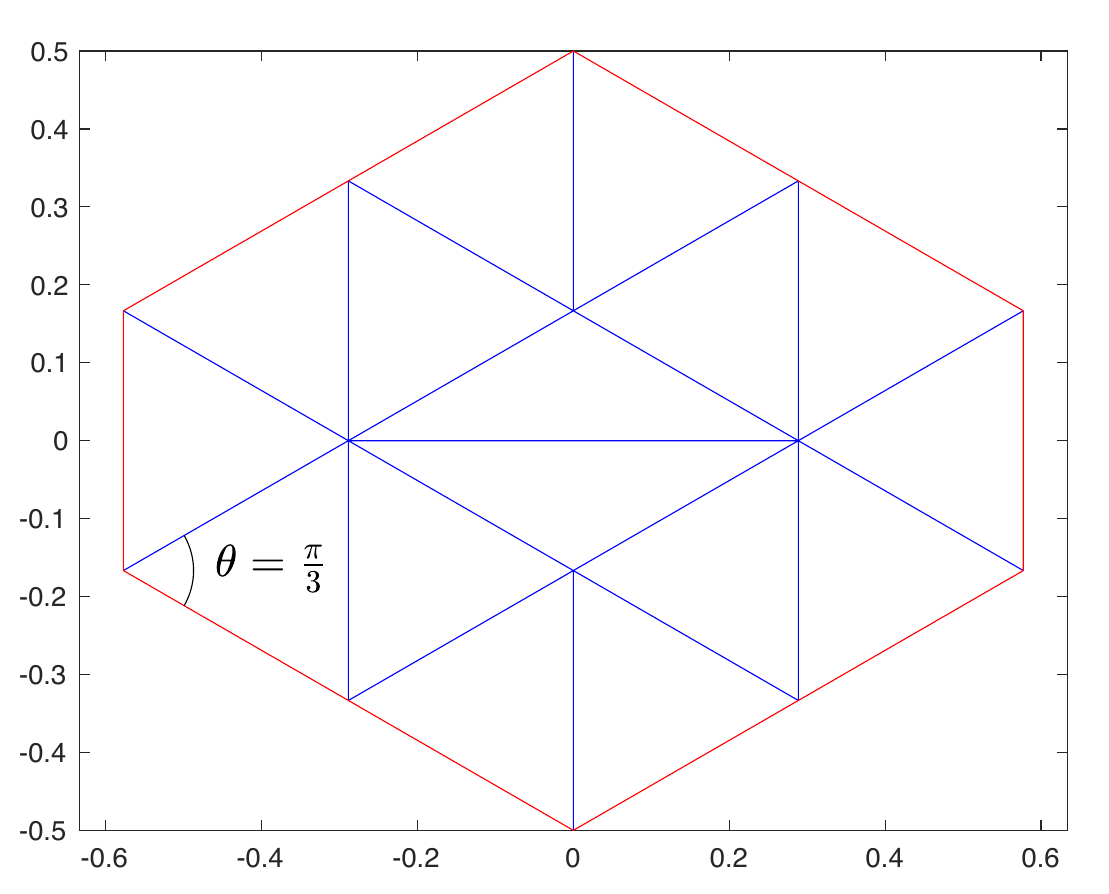}
\end{center}
\caption{The mesh $G_1(\pi/3)$  contains 1 edge violating the angle condition, but $\mat{A}^{-1}_1 >\mat{0}$ 
(verified numerically -- see also Fig.~\ref{fig:minGreensfunctions}).}
\label{fig:G1mesh}
\end{figure}
\begin{figure}[!h]
\begin{center}
        \includegraphics[width=3in]{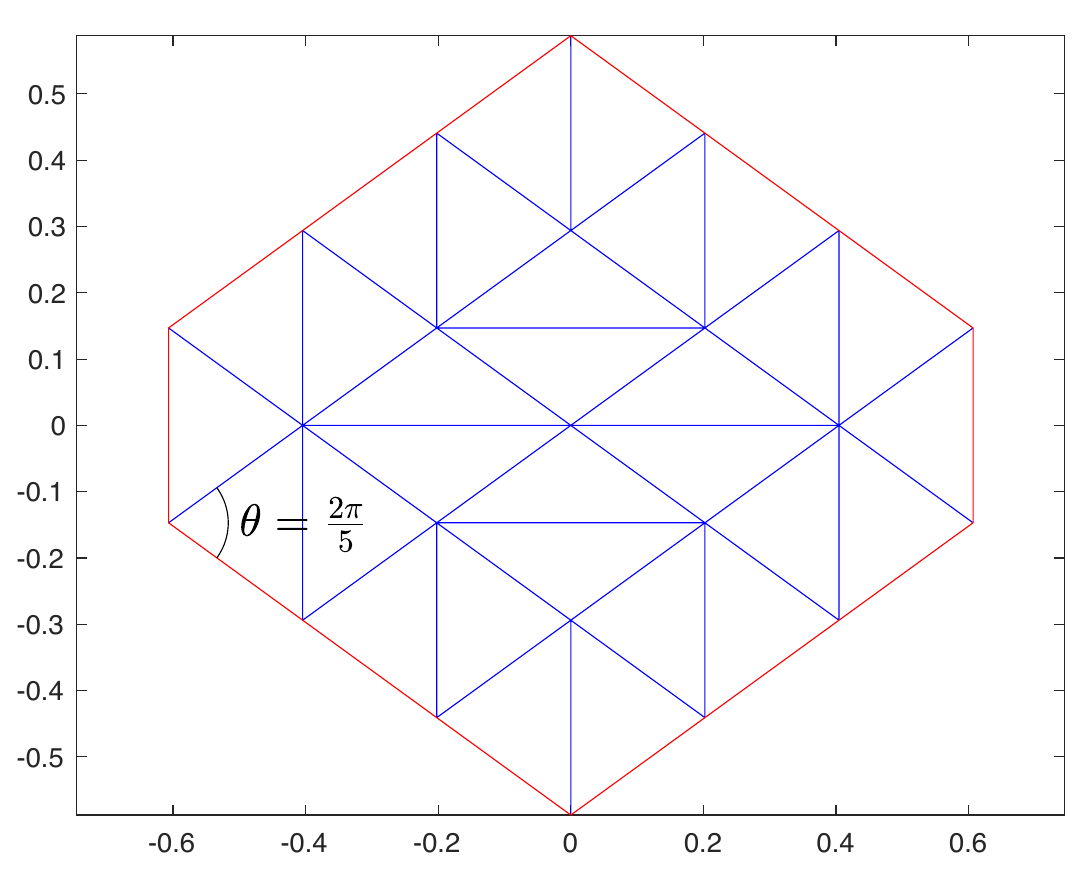}
\end{center}
\caption{The mesh $G_2(2\pi/5)$  contains 4 edges violating the angle condition, but $\mat{A}^{-1}_2 >\mat{0}$ 
(verified numerically -- see also Fig.~\ref{fig:minGreensfunctions}).}
\label{fig:G2mesh}
\end{figure}

\begin{example}
\label{ex:originates672072}

This example originates in~\cite{MR672072}; in the context of homogeneous Dirichlet boundary conditions, 
it is shown that for a certain $\varepsilon>0$ and 
$\pi/2-\varepsilon <\theta < \pi/2$, the discrete Green's function for $D(\theta)$
 cut into $n\times n$ rhombuses and along their long diagonal is positive, independent of 
the number of subdivisions $n$. Because in this setup the edges connecting to the boundary vertices
are defects, the wDMP is not satisfied. Variations on this example are further discussed in~\cite{MR2085400, MR3614014}.

In our case, the defects are grouped into ``mini'' versions of the domain ${D}(\theta)$, namely, they are
rhombuses made of $k\times k$ elemental (smallest) rhombuses, all of which are divided along their {\bf long} 
diagonal. Furthermore, we add  to these domains a one-layer lining
of elemental rhombuses that are divided along their short diagonal, and we eliminate the two extremal (on the $x$-axis) corner triangles. 
We call this mesh $G_k(\theta)$, and we omit $\theta$ when not necessary. 
In Fig.~\ref{fig:G1mesh} we show $G_1(\pi/3)$, and in Fig.~\ref{fig:G2mesh} we show $G_2(2\pi/5)$. Note that $G_k$ contains
$k^2$ defective edges. We should point out that defects of these types may arise naturally in mesh refinement, e.g.,  
when adding a midpoint to an edge and cutting the triangles adjacent along the medians {\color{black} -- see Fig.~\ref{fig:triangle_mariam}}. 
One can spot several 
such localized defects by carefully examining the refined meshes in Figure 6 from~\cite{MR4727060}. 

\begin{figure}[htb]
\centerline{\includegraphics[width=2in]{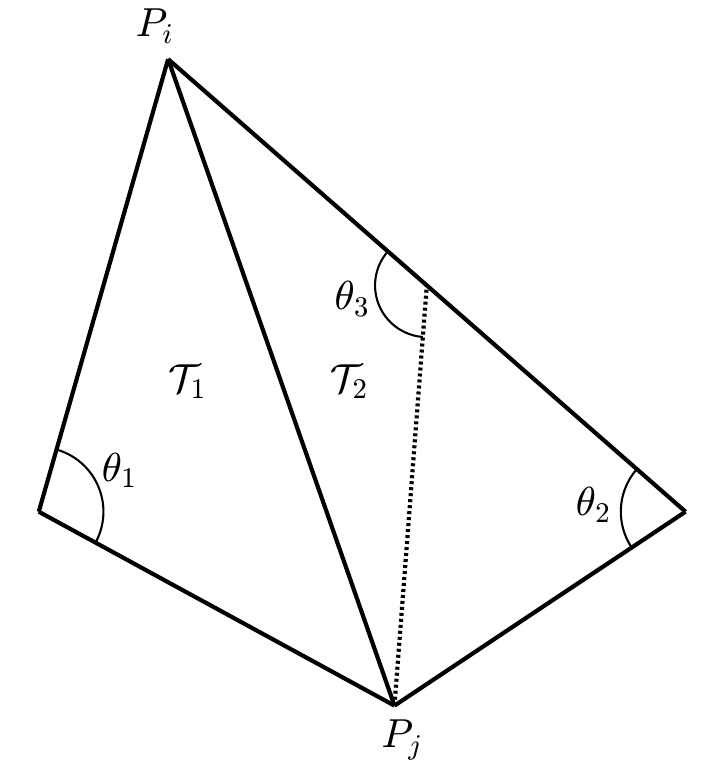}}
\caption{{\color{black}The angle condition holds for the original partition ${\op T}_1\cup {\op T}_2$, because 
  $\theta_1 + \theta_2 < \pi$. However, the bisection of an edge leads to $\theta_1 + \theta_3 > \pi$, implying
  that the angle condition is violated, as this causes $A_{ij}>0$ for the edge connecting $P_i$ and $P_j$.}
}
\label{fig:triangle_mariam}
\end{figure}

We now discuss the positivity of the discrete Green's function, i.e., the inverse of the stiffness matrix $\mat{A}_k(\theta)$
associated with the interior nodes of the mesh $G_k(\theta)$.
\begin{lemma}
\label{lma:invGreenposGk}
Denote by $\mat{A}_k(\theta)$ the stiffness matrix associated with the interior nodes of the mesh $G_k(\theta)$.
Then  for each $k\in \N$ there exists $0<\theta_k <\pi/2$ so that
\begin{equation}
\label{eq:invGreenposGk}
\mat{A}^{-1}_k(\theta) > \mat{0},\ \ \forall \ \ \theta_k < \theta \le \pi/2.
\end{equation}
\end{lemma}
\begin{proof}
The argument lies in the analysis of the limit case, $\theta=\pi/2$. 
A simple calculation using~\eqref{eq:formula2triangle} and~\eqref{eq:formula1triangle} shows  
that $\mat{A} = \mat{A}_k(\pi/2)$ is a scaled version of the matrix resulted from the
standard five-point stencil (finite difference) discretization of the Laplacian on the unit square with zero-boundary conditions:
\begin{eqnarray*}
\mat{A}_{ij} = \left\{
\begin{array}{cll}
4, 3,\  \mathrm{or}\ 2&\ \mathrm{if}\ i=j&\ \mathrm{and}\ P_i\ \ \mathrm{interior, \ side,\ or\ corner\ vertex,\ resp.;}\\
-1,&\ \mathrm{if}\ i\ne j&\ \mathrm{and}\ \overline{P_i P_j}\ \ \mathrm{is\ an\ edge\ with\ slope\ }\pm 1.
\end{array}
\right .
\end{eqnarray*}
Essentially, all the edges $\overline{P_i P_j}$ in $G_k(\pi/2)$ that {\color{black}are} parallel to 
the $x$- or the $y$-axis (the diagonals of the small squares) yield  $\mat{A}_{ij}=0$. 
It can easily be seen that $\mat{A}_k(\pi/2)$ is a Stieltjes matrix, meaning it is symmetric positive definite, 
satisfies~\eqref{eq:nonposent}, and is irreducible. Cf.~\cite{MR1753713},
\begin{eqnarray*}
\mat{A}^{-1}_k(\pi/2) > \mat{0}.
\end{eqnarray*}
The result follows simply by continuity of the map $\theta \mapsto \mat{A}^{-1}_k(\theta)$, namely we have 
\begin{eqnarray*}
\lim_{\theta \to \pi/2} \mat{A}^{-1}_k(\theta)  = \mat{A}^{-1}_k(\pi/2) > \mat{0},
\end{eqnarray*}
which then implies~\eqref{eq:invGreenposGk} for some $\theta_k<\pi/2$.
\end{proof}
In Fig~\ref{fig:minGreensfunctions} we plot the smallest value for $\mat{A}^{-1}_k(\theta)$ (computed numerically), 
for $k = 1, 2, 3, 4$, and we see that is turns positive as 
$\theta\to \pi/2$, thus supporting Lemma~\ref{lma:invGreenposGk}. 
It appears that  $\theta_1 < \theta_2 < \theta_3 < \theta_4$, although this is not essential for our arguments.
\end{example}

\begin{figure}[!h]
\begin{center}
        \includegraphics[width=5.2in]{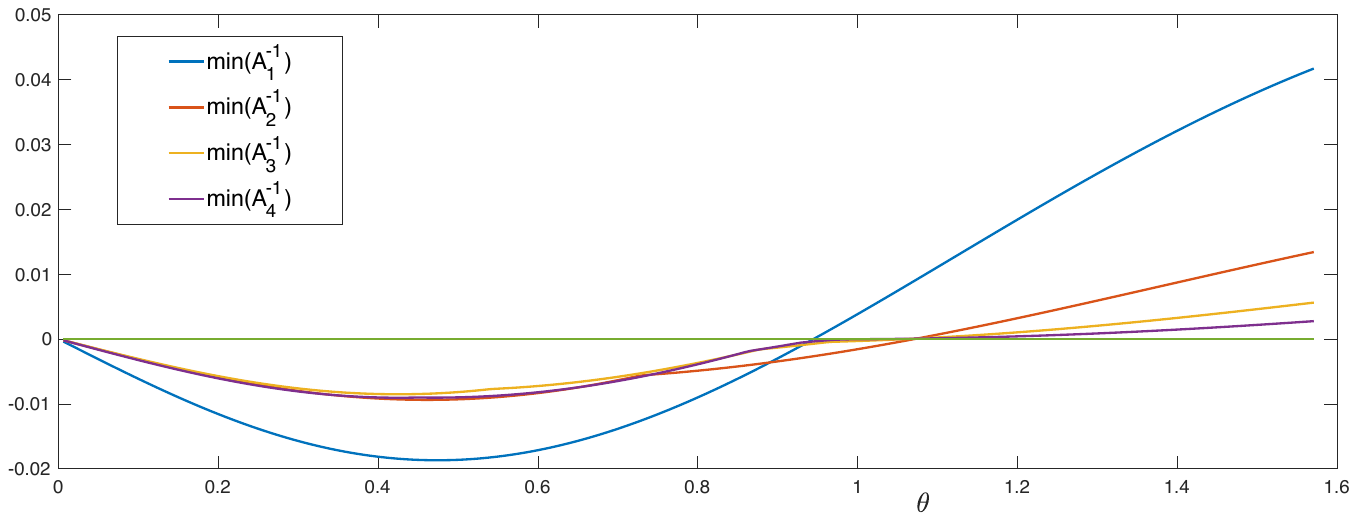}
\end{center}
\caption{{\color{black}The smallest values of $\mat{A}^{-1}_1(\theta),\dots,\mat{A}^{-1}_4(\theta)$ are shown as functions of $\theta$ (on the $x$-axis). 
Note that each of the four functions} {\color{black}has a root $\theta_k <\pi/2$ so that it is positive for $\theta_k< \theta < \pi/2$.}}
\label{fig:minGreensfunctions}
\end{figure}

\subsection{More complicated meshes}

The next result describes a set of meshes on ${D}(\theta)$ for which various DMPs hold as $h\to 0$.
\color{black}
\begin{theorem}
\label{thm:example1}
Let $k\in \N$ and $\theta$ so that
$\max_{\ell=1}^k\theta_{\ell} < \theta <\frac{\pi}{2}$. Consider a uniform partition of $D=D(\theta)$ in $n\times n$ rhombuses,
with each small rhombus being divided along either its long diagonal or its short diagonal, and let $h$ be the mesh size.
We assume that 
each interior vertex $P$ lies in the interior of a mesh that is a scaled version of $G_{\ell}(\theta)$ for some
$1\le \ell\le k$, { or} the star  around $P$ has only acute angles. (Examples of such meshes are given 
in Figures~\ref{fig:G1meshdefect}-\ref{fig:G2meshdefect}). Let $(\op{S}^E_h)_{E\subseteq D}$ be 
the {consistent family of solution operators} associated with the discrete linear variational {\color{black}problem}~\eqref{eq:weakelldisc}.
Then the following hold:\\
{\bf (i)} If $\tilde{c}\equiv 0$, then the global solution operator $\op{S}^{\tilde{D}}_h$ satisfies sDMP-A, and $\op{S}^{{D}}_h$
satisfies wDMP-A.\\
{\bf (ii)}  If $0\le \tilde{c}\le \tilde{c}_{\max}$, with $\tilde{c}_{\max}$ a constant, 
then there exists $h_{\max}>0$ depending on $\theta$ and $\tilde{c}_{\max}$ so that  for $0< h< h_{\max}$
the global solution operator $\op{S}^{\tilde{D}}_h$ satisfies sDMP-B, and $\op{S}^{D}_h$
satisfies wDMP-B.
\end{theorem}
\begin{proof}
For {\bf (i)}, the hypothesis implies that each interior vertex $P$ lies either in the interior of a patch $F_P$ that is
a rescaled version of $G_{\ell}(\theta)$, or in a star $E_P$ (the union of the triangles having $P$ as vertex) that has only acute angles.
In the former case, since the stiffness matrix in 2D is independent of the scale, Lemma~\ref{lma:invGreenposGk} shows 
$\mat{A}^{-1}_{\ell} = \mat{A}^{-1}_{\ell}(\theta) > \mat{0}$, hence the same holds for the stiffness matrix $\mat{A}_P$
associated with the interior vertices of $F_P$, because $\mat{A}_P$ {\bf is identical to} $\mat{A}_{\ell}$.
Now let $\mat{A}^b_{\ell}=\mat{A}^b_{\ell}(\theta)$ be the stiffness matrix connecting interior and boundary entries in $G_{\ell}(\theta)$,
and $\mat{A}^b_P$ be its analogue on $F_P$.
Since the boundary layer of $F_P$ contains only edges that satisfy~\eqref{eq:anglecond}, we have $(\mat{A}^b_P)_{ij}<0$ for every edge 
$\overline{P_i P_j}$ with $P_i\in Int(F_P)$ and $P_j\in \partial F_P$. This shows that
$\mat{A}_P^{-1} \mat{A}^b_P < \mat{0}$,
that is,~\eqref{eq:matformsdmpAthm} is verified for $F_P$. 
For the latter case, when the patch is the star $E_P$, the same argument as in Theorem~\ref{thm:classicalDMP} shows 
that~\eqref{eq:matformsdmpAthm} is verified on $E_P$ as well. Corollary~\ref{cor:sDMPmatform}{\bf (i)} shows that
$\op{S}^{\tilde{D}}_h$ satisfies sDMP-A. As in Example~\ref{ex:threelinesmesh}, adding the two corner triangles 
leads to $\op{S}^{{D}}_h$ satisfying wDMP-A.\\
The argument for {\bf (ii)} lies in the scaling of the mass matrix compared to that of the stiffness matrix.
Assume a point  $P$ lies in the interior of a patch $F_P$  that is
a rescaled version of $G_{\ell}(\theta)$, as above. Denote by $\mat{M}_{\ell}=\mat{M}_{\ell}(\theta)$ the mass matrix 
associated with the interior vertices of the reference mesh $G_{\ell}(\theta)$. Due to scaling (see~\eqref{eq:formula4triangle}),
the associated mass matrix for $F_P$ is $K h^2\mat{M}_{\ell}$, with $K$ a dimensionless constant depending on $\theta$ and $\ell$.
Similarly, if {\color{black}$\mat{M}^b_{\ell}=\mat{M}^b_{\ell}(\theta)$} is the mass matrix connecting interior and boundary entries
on $G_{\ell}(\theta)$, then the analogous matrix on $F_P$ is $K h^2\mat{M}^b_{\ell}$. Hence, the interior and boundary reaction matrices 
$\mat{C}_P$ and $\mat{C}^b_P$, respectively, associated with $F_P$ (see~\eqref{eq:stiffmassreac}) satisfy
\begin{equation}
\label{eq:masstozero}
\mat{0}\le \mat{C}_P \le K \tilde{c}_{\max} h^2\mat{M}_{\ell},\ \ \ \mat{0}\le \mat{C}^b_P \le K \tilde{c}_{\max} h^2\mat{M}^b_{\ell}.
\end{equation}
Consequently
{\color{black}
\begin{eqnarray*}
\lim_{h\to 0} (\mat{A}_P + \mat{C}_P)^{-1} = \mat{A}_P^{-1} > \mat{0},
\end{eqnarray*}
and
\begin{eqnarray*}
\lim_{h\to 0} (\mat{A}_P + \mat{C}_P)^{-1}(\mat{A}^b_P + \mat{C}^b_P) = \mat{A}_P^{-1}\mat{A}^b_P < \mat{0}.
\end{eqnarray*}
}
Hence, there exists $h_{\max}>0$ depending on $\tilde{c}_{\max}$ and $K$ (which depends  on $\theta$ and $k$ - the largest index of the defects),
so that for $0<h<h_{\max}$ we have
\begin{eqnarray}
\label{eq:Cconditionsex1}
(\mat{A}_P + \mat{C}_P)^{-1}  > \mat{0},\ \ \mathrm{and}\ \ 
(\mat{A}_P + \mat{C}_P)^{-1}(\mat{A}^b_P + \mat{C}^b_P)  < \mat{0}.
\end{eqnarray}
If, instead, the star $E_P$ around $P$ has only acute angles, similar inequalities hold if we replace 
$F_P$ with $E_P$.
By Corollary~\ref{cor:sDMPmatform}{\bf (ii)}, satisfies $\op{S}^{\tilde{D}}_h$ satisfies sDMP-B. 
The argument shown in Example~\ref{ex:threelinesmesh} can be replicated to show that
when adding back the two triangles at the {\color{black}horizontal} extremes of $\tilde{D}$, the global solution operator $\op{S}^{{D}}_h$ satisfies wDMP-B. 
\end{proof}

\begin{figure}[!h]
\begin{center}
        \includegraphics[width=5.2in]{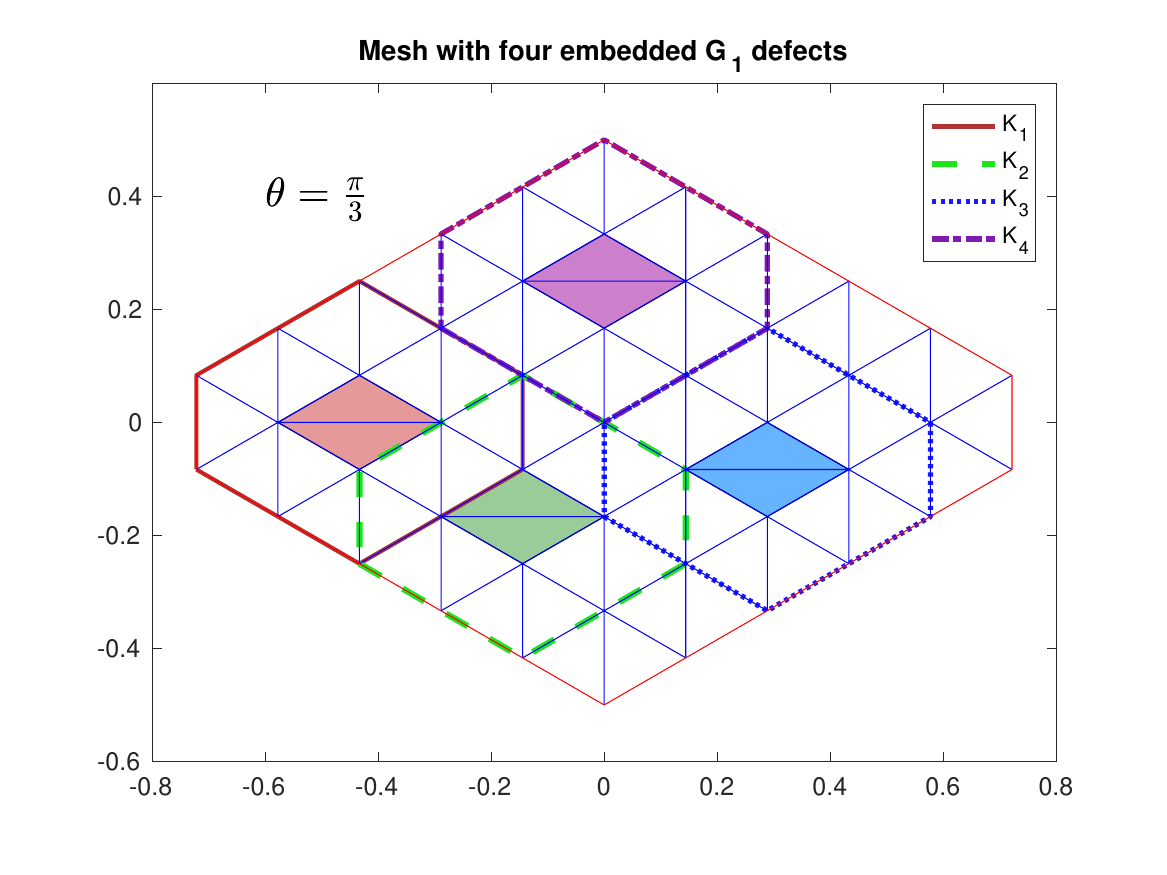}
\end{center}
\caption{Mesh on $\tilde{D}(\pi/3)$ with embedded $G_1$ defects ($k=1$ in reference to Theorem~\ref{thm:example1}). The four domains which are similar
  to $G_1$ are marked as $K_1, \dots, K_4$.}
\label{fig:G1meshdefect}
\end{figure}
\begin{figure}[!h]
\begin{center}
        \includegraphics[width=5.2in]{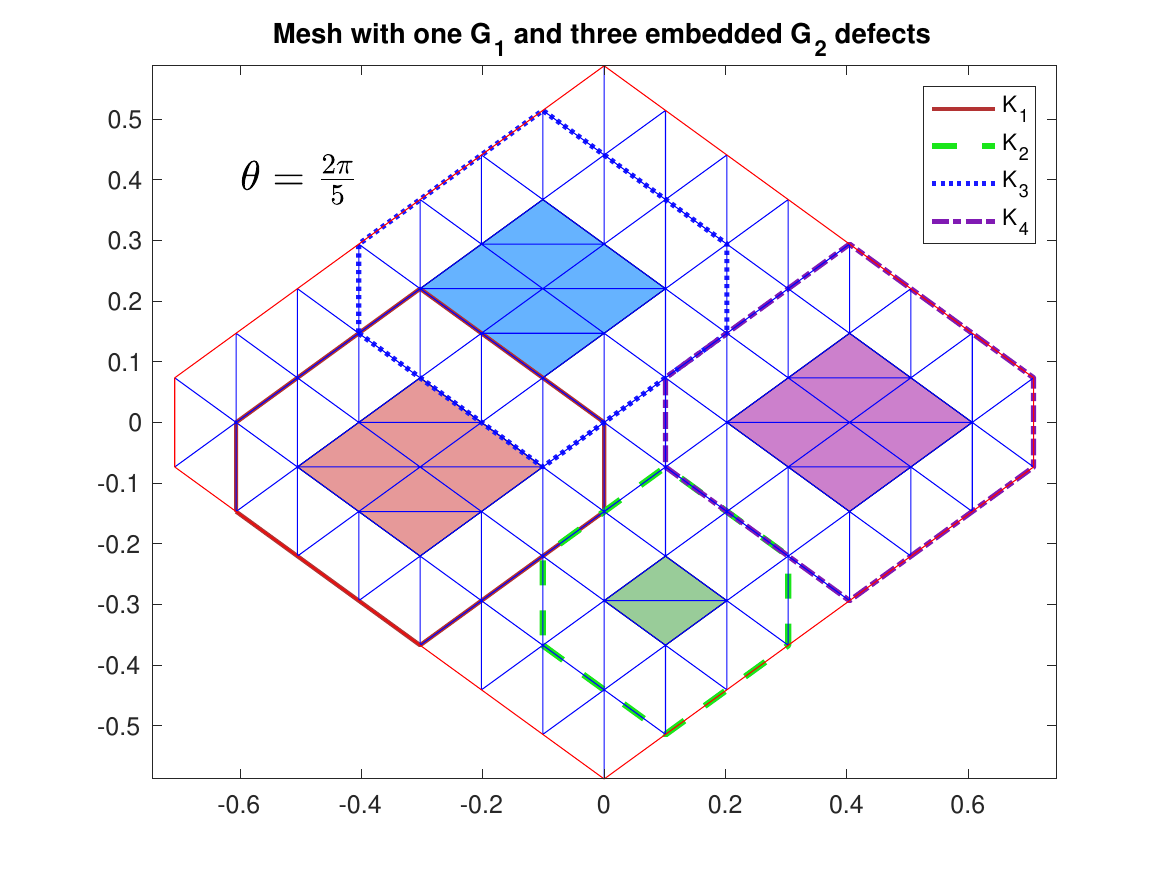}
\end{center}
\caption{Mesh on $\tilde{D}(2\pi/5)$ with embedded $G_1$ and $G_2$ defects ($k=2$ in reference to Theorem~\ref{thm:example1}). The four domains which are similar
  to $G_1$ or $G_2$ are marked as $K_1, \dots, K_4$.}
\label{fig:G2meshdefect}
\end{figure}



\subsection{
wDMP-A for meshes with embedded nearly degenerate triangles}
\label{ssec:degeneratemesh}
\begin{figure}[!h]
\begin{center}
        \includegraphics[width=3.2in]{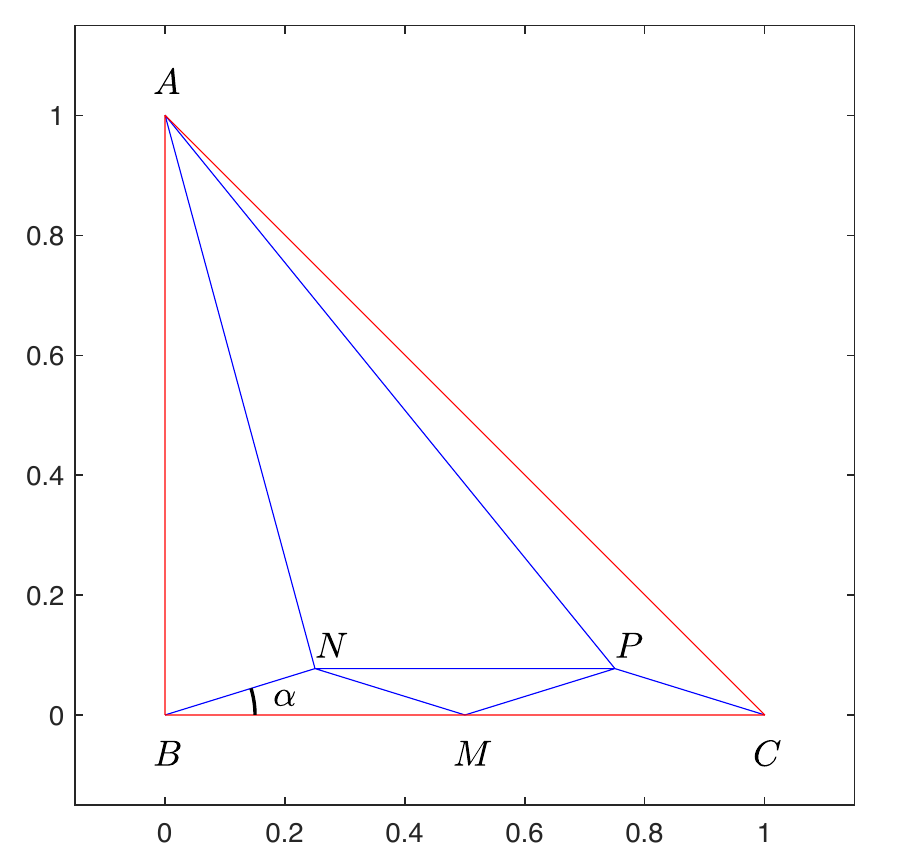}
\end{center}
\caption{Triangle with nearly degenerate internal triangulation.}
\label{fig:degenerate_triangle}
\end{figure}
In this section we prove that the finite element solution operator
for the Poisson equation satisfies the wDMP-A for a nearly degenerate mesh. This mesh
is closely related to the example given in Section~5 in~\cite{MR2085400}, 
in which a quasi-uniform triangulation is shown to violate the DMP as the mesh size $h\to 0$, in the sense that
the discrete Green's function has provably negative values for all $h>0$. 
The key element in both situations is the reference triangle $\Delta A B C$,
with $A=(0,0), B=(0,1), C=(1,0)$, to which the following three points are added:
\begin{eqnarray*}
&&M=\left(\frac{1}{2}, 0\right),\ \ 
  N=\left(\frac{1}{4}, \frac{\tan \alpha}{4}\right),\ \  
  P=\left(\frac{3}{4}, \frac{\tan \alpha}{4}\right).
\end{eqnarray*}
Together they give rise to the mesh shown in Fig.~\ref{fig:degenerate_triangle}. The mesh on $\Delta A B C$ violates the angle
condition
due to the fact that $a(\phi_M, \phi_P)>0$ for a small enough angle $\alpha>0$.
In~\cite{MR2085400} $\Delta A B C$ is placed at the boundary of a
uniform three-line mesh similar to the one in Example~\ref{ex:threelinesmesh} with $\theta=\pi/2$, 
and it is shown that for a fixed, sufficiently small 
angle~$\alpha$ (refer to Fig.~\ref{fig:degenerate_triangle}), the discrete Green's function satisfies $g_h^P(N)<0$
for all $h>0$. By contrast, here we show that by placing $\Delta A B C$ in the same type of mesh, just one element away from the boundary 
as in Fig.~\ref{fig:mesh_embed_triangle}, the wDMP-A is satisfied for any $0< \alpha < \alpha_0$ for a certain  $\alpha_0>0$. In fact, it can be placed
anywhere inside the mesh, as long as it is surrounded by a layer of ``good'' triangles. It is instructive to compare the discrete Green's functions
between the cases when $\Delta A B C$ is one layer inside a domain Fig.~\ref{fig:degenerate_el_comparison} (left) vs. right at the 
boundary Fig.~\ref{fig:degenerate_el_comparison} (right). We see in the bottom left corner of Fig.~\ref{fig:degenerate_el_comparison}
that $g^P_h >0$ inside the domain, and it has the expected shape of a Green's function, while in the bottom right we have that $g^P_h(N)<0$ (note the difference
in the scales as well).

\begin{figure}[!h]
\begin{center}
        \includegraphics[width=4.5in]{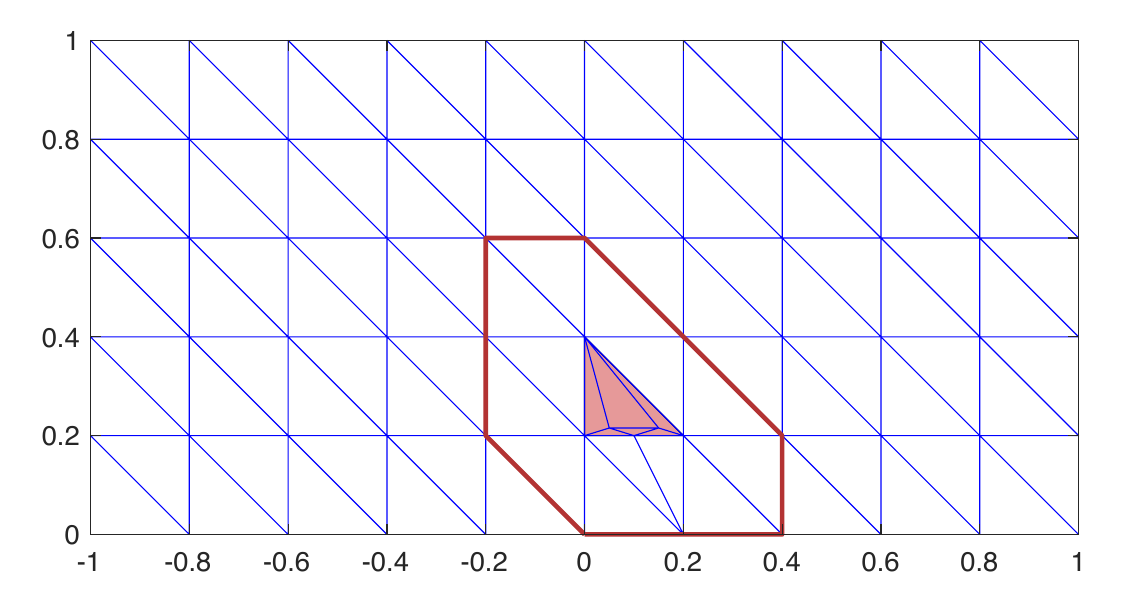}
\end{center}
\caption{Uniform three-line mesh with an embedded, irregularly divided triangle one layer away from the boundary. The finite element solver
for the Poisson equation satisfies the wDMP-A on this mesh, but {\color{black}fails} to do so if the triangle is placed at the boundary.}
\label{fig:mesh_embed_triangle}
\end{figure}

\begin{figure}[!h]
\begin{center}
        \includegraphics[width=5.2in]{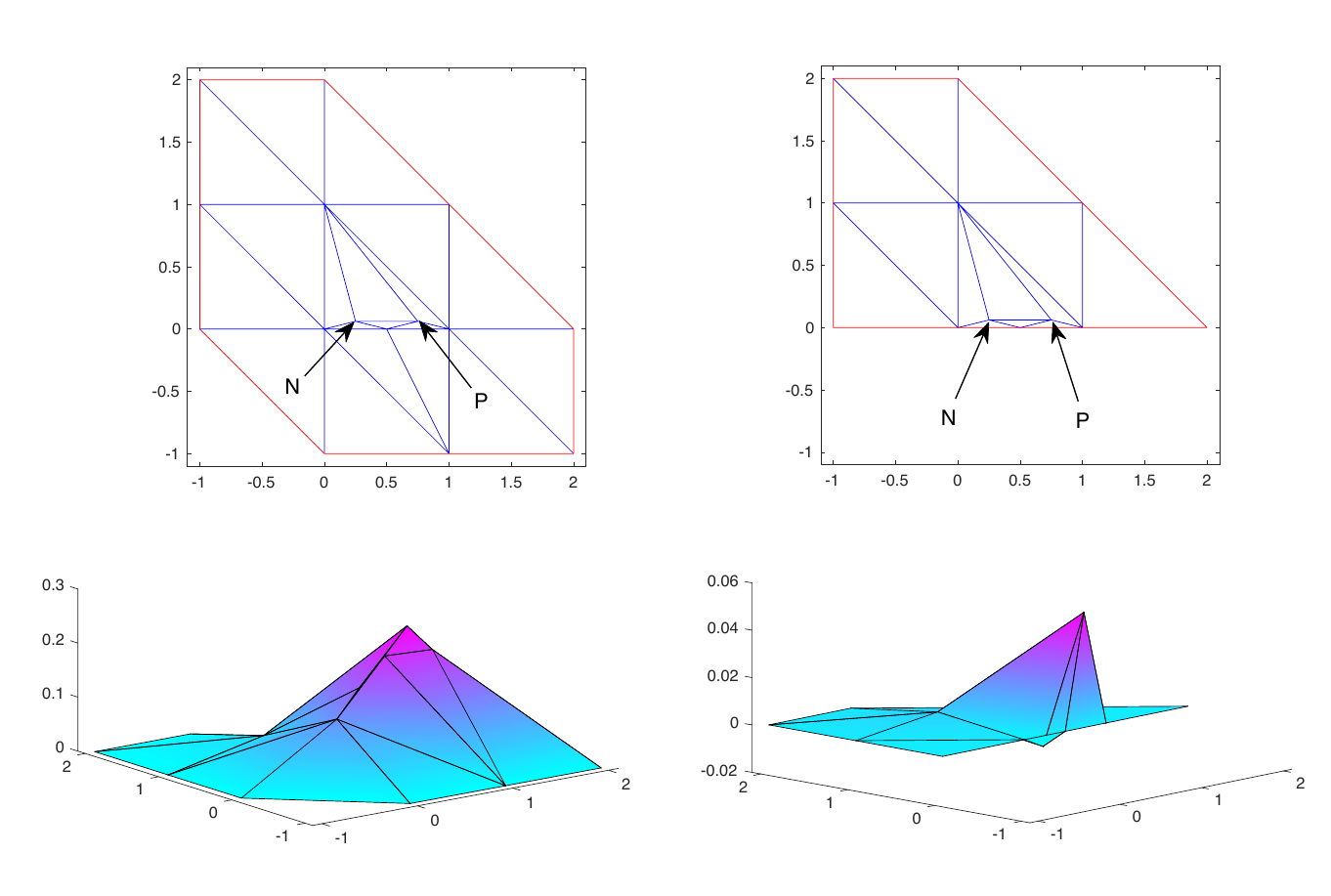}
\end{center}
\caption{Green's function for a mesh with near-degenerate elements {\bf near} the boundary (left) vs. {\bf at} the boundary (right); 
in the first case we have $g_h^P(V)>0$ for any internal vertex $V$,  while in the second case  $g_h^P(N)<0$.}
\label{fig:degenerate_el_comparison}
\end{figure}
The main argument is rooted in the analysis of the stiffness matrix for the subdomain $E$
pictured in the top left corner of Fig.~\ref{fig:degenerate_el_comparison}
(also in Figure~\ref{fig:patch_around_triangle}), representing the union of the supports of the nodal basis 
functions associated with the vertices $B, C, A, M, N, P$, {\bf in this order}. Thus the ordered basis in
the finite element space $\op{V}$ of this subdomain is
$\op{B} = \{\varphi_B, \varphi_C, \varphi_A, \varphi_M, \varphi_N, \varphi_P\}$; we define 
$\op{V} = \mathrm{span}(\op{B})$.
The associated stiffness matrix is denoted by $\mat{S}=\mat{S}(\alpha)$ (e.g., $\mat{S}_{11} = a_E(\varphi_B, \varphi_B)$, 
$\mat{S}_{35} = a_E(\varphi_A, \varphi_N)$, etc). The key result is the following.
\begin{theorem}
\label{thm:degenerate_triangle}
There exists a singular, {\color{black}rank-3 deficient} matrix ${\mat{T}_0}>\mat{0}$ so that
\begin{equation}
  \label{eq:limstiffex2}
  \lim_{\alpha\to 0} \mat{S}^{-1}(\alpha) = \mat{T}_0. 
\end{equation}
\end{theorem}
The proof of Theorem~\ref{thm:degenerate_triangle} is technical, but elementary; it is given in Appendix~\ref{sec:appendix}, 
where also the precise
value for $\mat{T}_0$ is also shown. The theoretical result above was validated numerically.
{\color{black} With regards to DMPs, the most remarkable part of Theorem~\ref{eq:limstiffex2} is the positivity of 
$\mat{T}_0$ as $\alpha \to 0$. The existence of the limit as a bounded and singular matrix is related to the behavior of 
discrete Green's functions on degenerate meshes. In order to better illustrate this behavior
we provide a simpler case in Appendix~\ref{sec:appendix2}, which
emerged from an example discussed in~\cite{MR2085400} where the DMP holds on certain meshes violating the angle condition.}

\begin{corollary}
\label{cor:posSinv}
There exists $\alpha_0>0$ so that for $\alpha\in (0,\alpha_0)$ we have $\mat{S}^{-1} > \mat{0}$.
\end{corollary}
\begin{remark}
\label{rem:globalmaxalpha}
Numerical results show that $\mat{S}^{-1} > \mat{0}$ {\bf for all} $\alpha>0$ for which the mesh is non-degenerate 
(see Figure~\ref{fig:smallest_entryDG}).
\end{remark}
\begin{figure}[!h]
\begin{center}
        \includegraphics[width=4.2in]{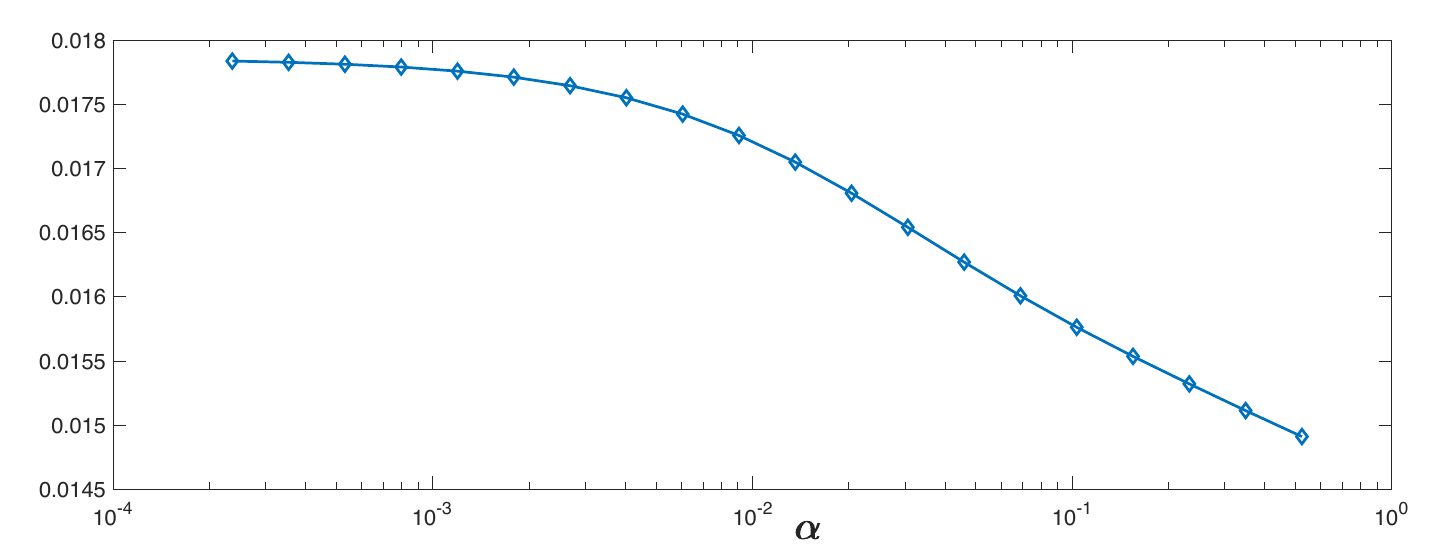}
\end{center}
\caption{Smallest entry in $\mat{S}(\alpha)^{-1}$ for $2.36 \times 10^{-4} <  \alpha \le \pi/6$.}
\label{fig:smallest_entryDG}
\end{figure}

\begin{corollary}
\label{cor:wDMPdegenerate}
Consider the three-line mesh shown in Fig.~\ref{fig:mesh_embed_triangle} on a rectangle. We further partition
any number of triangles so that the local mesh on each triangle is a scaled version of the mesh on $\Delta ABC$ with $0<\alpha<\alpha_0$.
Moreover, assume any two such triangles are separated from each other and from the boundary by one layer of unpartitioned triangles.
Then the finite element solution operator of the Poisson equation satisfies the wDMP-A on this mesh.
\end{corollary}
\begin{proof}
The statement follows from Corollary~\ref{cor:posSinv} and Theorem~\ref{thm:wdmpa}. The hypotheses ensure that
every vertex $V$ lies in the interior of a patch similar to $E$ in the top left corner of Fig.~\ref{fig:degenerate_el_comparison},
or the star around $V$ is similar to a standard star in  Example~\ref{ex:threelinesmesh}; in either case, the inequalities~\eqref{eq:matformWdmpAthm}
are satisfied.
\end{proof}

\section{The semilinear problem}
\label{sec:semilinear}
In this section we prove a result for semilinear elliptic problems equations which is similar to Theorem~\ref{thm:classicalDMP}. 
While results of this type can be found in a number of
articles~\cite{MR2121074, MR2226992, MR2392463, MR2914278, MR2991837, MR3425305, MR4092285, MR4848416}, we will show that our technique 
is also applicable to prove sDMPs for semilinear elliptic equations. However, in this case we do assume off-diagonal entries
of the stiffness matrix to be negative.
\begin{theorem}
\label{thm:semilinear_sDMP-B}
Consider the discrete semilinear elliptic equation under  the formulation~\eqref{eq:weakelldiscinterp}, 
and let {\color{black}${c}$ satisfy~\eqref{eq:Cnondecreasing}--\eqref{eq:cincrease}}. 
Furthermore, assume there is a  constant $C_A>0$ independent of $h$ so
that whenever vertices $P_i$ and $P_j$ are adjacent with at least one of them lying in the interior 
(that is, $\overline{P_i P_j}$ is not a boundary edge), we have $\mat{A}_{ij}<0$ and 
\begin{equation}
\label{eq:stiffmatrixlimitsrat}
\mat{M}_{ij}\le C_A h^2 |\mat{A}_{ij}|.
\end{equation}
Then there exists $h_{\max}>0$ so that for all $0<h<h_{\max}$, 
the discrete solution operator $\op{S}^E_h$ of~\eqref{eq:weakelldiscinterp} satisfies sDMP-B for every connected discrete subdomain 
$E\subseteq D$.
\end{theorem}
\begin{proof}
{\color{black}The} argument follows the same path as in Theorem~\ref{thm:classicalDMP}, namely we prove that $\op{S}^E_h$ satisfies 
the sDMP-B for any star $E_P$ around a vertex $P$. The conclusion will follow from Corollary~\ref{cor:DMP}.

Thus, it is sufficient to prove the prove sDMP-B assuming the mesh has only one interior point $P_1$ and boundary points
$P_2, \dots P_N$, i.e., $D=E_{P_1}$. In this case $h \le \mathrm{diam}(D) \le 2 h$; note that the mesh size plays a role in the
estimation. Let $u_h=\op{S}_h^D(f,u_h^b)$, and
denote $\mat{F}_i = f(P_i)$, $\mat{U}_i = u_h(P_i)$, for $1\le i\le N$. In matrix terms,~\eqref{eq:weakelldiscinterp} is represented by the nonlinear 
scalar equation: find $\mat{U}_1 \in \R$ so that
\begin{align}
\label{eq:sldiscmat}
\mat{A}_{11} \mat{U}_1 + \sum_{i=2}^N \mat{A}_{i1}\mat{U}_i + \mat{M}_{11} \:c(P_1, \mat{U}_1) + \sum_{i=2}^N \mat{M}_{i1}\:c(P_i, \mat{U}_i)
= \sum_{i=1}^N \mat{M}_{i1} \mat{F}_i.
\end{align}
Since $f\ge 0$, we have 
\begin{align}
\label{eq:sldiscmatineq}
\mat{A}_{11} \mat{U}_1 + \sum_{i=2}^N \mat{A}_{i1}\mat{U}_i + \mat{M}_{11} \:c(P_1, \mat{U}_1) + \sum_{i=2}^N \mat{M}_{i1}\:c(P_i, \mat{U}_i) \ge 0.
\end{align}
Note that $\mat{A}_{11}>0$, and $\mat{M}_{i1}>0$ for $1\le i \le N$.
We consider two cases: \\
{\bf Case 1:} First assume that $\mat{U}_i\ge 0$ for $2\le i\le N$, which implies that 
\begin{equation}
\label{eq:posClip}
0\le c(P_i, \mat{U}_i)\le L_c\: \mat{U}_i,\ \ 2\le i\le N.
\end{equation}
In this case we have 
$$
-\max_{\partial D}{u^-_h} = -\max_{2\le i\le N} \mat{U}_i^- = 0.$$ 
If we had $\mat{U}_1<0$, then $c(P_1, \mat{U}_1)\le 0$. Hence,
\begin{eqnarray}
\nonumber
&&0<\mat{A}_{11} (-\mat{U}_1)  - \mat{M}_{11} \:c(P_1, \mat{U}_1) \stackrel{\eqref{eq:sldiscmatineq}}{\le}
\sum_{i=2}^N \mat{A}_{i1}\mat{U}_i + \sum_{i=2}^N \mat{M}_{i1}\:c(P_i, \mat{U}_i)\\
\nonumber
&&\stackrel{\eqref{eq:posClip}}{\le} \sum_{i=2}^N (\mat{A}_{i1} + L_c\mat{M}_{i1})\mat{U}_i
=\sum_{i=2}^N (-\mat{A}_{i1})(-1 + L_c\mat{M}_{i1}/(-\mat{A}_{i1}))\mat{U}_i\\
\label{eq:case1ineq}
&&\stackrel{\eqref{eq:stiffmatrixlimitsrat}}{\le}
\sum_{i=2}^N (-\mat{A}_{i1})(-1 + L_c C_A h^2)\mat{U}_i.
\end{eqnarray}
This is not possible if $h<h_0 = \sqrt{(L_c C_A)^{-1}}$, for in that case we have
\begin{equation}
\label{eq:case1ineq2}
(-\mat{A}_{i1})(-1 + L_c C_A h^2) <  0,\ \ \forall i=2,\dots,N.
\end{equation}
Together with $\mat{U}_i\ge 0$ for $2\le i\le N$,  this renders the sum in~\eqref{eq:case1ineq} to be nonpositive, contradicting our assumption on 
$\mat{U}_1$. It remains that $\mat{U}_1\ge 0$. 
For the strong part, assume $\mat{U}_1 = 0$; then we also have $c(P_1, \mat{U}_1) = 0$. 
Using the same line of arguments as above, it follows that the sum from~\eqref{eq:case1ineq} is nonnegative. However,
since all the terms in~\eqref{eq:case1ineq}
are nonpositive, they all must be zero, forcing $\mat{U}_i = 0$ for $i=2,\dots,N$.
 \\
{\bf Case 2:} We assume that at least one of the values $\mat{U}_i$, $i=2, \dots N$ is negative. 
We partition $I_N = \{2,\dots,N\} = \op{N}_{-}\cup \op{N}_{+}$, with
\begin{align*}
\op{N}_{-} = \{i\in I_N :\ \mat{U}_i< 0\},\ \ \op{N}_{+} = \{i\in I_N :\ \mat{U}_i\ge 0\}.
\end{align*}
This implies $c(P_i, \mat{U}_i)\le 0$ for $i\in \op{N}_{-}$, 
and  $c(P_i, \mat{U}_i)\ge 0$ for $i\in \op{N}_{+}$.
Note that
\begin{align}
\label{eq:acoeffineq}
-\mat{A}_{11} - \sum_{i\in \op{N}_{-}} \mat{A}_{i1} = \sum_{i\in \op{N}_{+}} \mat{A}_{i1} \le 0,
\end{align}
because $\sum_{i = 1}^N\mat{A}_{1i}=0$. Note that all the arguments still hold if $\op{N}_+=\varnothing$. 
If $\mat{U}_1 < \min_{i\in \op{N}_{-}} \mat{U}_i<0$, then cf.~\eqref{eq:sldiscmatineq} we have
\begin{align*}
& \sum_{i\in \op{N}_{+}} \mat{A}_{i1}\mat{U}_i + \mat{M}_{11} \:c(P_1, \mat{U}_1) + \sum_{i=2}^N \mat{M}_{i1}\:c(P_i, \mat{U}_i) \ge 
-\mat{A}_{11} \mat{U}_1 + \sum_{i\in \op{N}_{-}} (-\mat{A}_{i1})\mat{U}_i\\
& =
\underbrace{(-\mat{A}_{11} -\sum_{i\in \op{N}_{-}} \mat{A}_{i1}) \mat{U}_1}_{\ge 0} + \sum_{i\in \op{N}_{-}} (-\mat{A}_{i1})(\mat{U}_i-\mat{U}_1)\ge 
\sum_{i\in \op{N}_{-}} (-\mat{A}_{i1})(\mat{U}_i-\mat{U}_1) >0.
\end{align*}
Therefore, 
\begin{align}
\nonumber
&0< \mat{M}_{11} \:c(P_1, \mat{U}_1)  + \sum_{i\in \op{N}_{+}} \left(\mat{A}_{i1}\mat{U}_i+\mat{M}_{i1}\:c(P_i, \mat{U}_i) \right) +
\sum_{i\in \op{N}_{-}} \mat{M}_{i1}\:c(P_i, \mat{U}_i)\\
\label{eq:sldiscmatineqcase3}
&\stackrel{\eqref{eq:stiffmatrixlimitsrat}}{\le}
\mat{M}_{11} \:c(P_1, \mat{U}_1)  +
\sum_{i\in \op{N}_{+}} (-\mat{A}_{i1})(-1 + L_c C_A h^2)\mat{U}_i +
\sum_{i\in \op{N}_{-}} \mat{M}_{i1}\:c(P_i, \mat{U}_i).
\end{align}
Following~\eqref{eq:case1ineq2}, all the terms in the sum of~\eqref{eq:sldiscmatineqcase3} are nonpositive for $h<h_0$, thus contradicting the assumption
on $\mat{U}_1$. Hence,
\begin{equation}
\label{eq:minu0}
\mat{U}_1 \ge  \min_{i\in \op{N}_{-}} \mat{U}_i = -\max_{\partial E_{P_1}} u_h^-.
\end{equation}
If equality holds in~\eqref{eq:minu0}, then we have $\mat{U}_i-\mat{U}_1 \ge 0$, $c(P_i, \mat{U}_i)-c(P_i, \mat{U}_1) \ge 0$, and 
$c(P_i, \mat{U}_1)\le 0$ for $1\le i \le N$. Consequently,
\begin{align*}
\label{eq:sldiscmatineq4}
&0\le \mat{A}_{11}\mat{U}_1 + \sum_{i=2}^N \mat{A}_{i1}\mat{U}_i +  \sum_{i=1}^N\mat{M}_{i1}\:c(P_i,\mat{U}_i)  \\
&= \sum_{i=2}^N \mat{A}_{i1}(\mat{U}_i-\mat{U}_1) + \sum_{i=1}^N\mat{M}_{i1}\:c(P_i,\mat{U}_1) + 
\sum_{i=2}^N\mat{M}_{i1}\:\left(c(P_i,\mat{U}_i) - c(P_i, \mat{U}_1)\right)\\
&\le \sum_{i=2}^N (\mat{A}_{i1}+ L_c \mat{M}_{i1})(\mat{U}_i-\mat{U}_1) + \sum_{i=1}^N\mat{M}_{i1}\:c(P_i,\mat{U}_1)\\
&\le \sum_{i=2}^N (-\mat{A}_{i1})(-1 + L_c C_A h^2 )(\mat{U}_i-\mat{U}_1) + \sum_{i=1}^N\mat{M}_{i1}\:c(P_i,\mat{U}_1).
\end{align*}
Since all the terms in the sum above are nonpositive for $h<h_0$, it follows that they are all zero, showing that $\mat{U}$ is constant 
and that $c(P_i,\mat{U}_1) = 0$ for $1\le i\le N$.
\end{proof}

\begin{remark}
\label{re:scalignmass}
Condition~\eqref{eq:stiffmatrixlimitsrat} is related to the scaling of the mass matrix entries vs. those of the stiffness matrix: when scaling an
element in $\R^d$ by a factor $h$, the mass matrix scales by $h^d$, and the stiffness matrix by $h^{d-2}$.
Therefore~\eqref{eq:stiffmatrixlimitsrat} holds on the
all uniform meshes $G_k(\theta)$ in Section~\ref{ssec:nonmonotedge}, with $C_A$ depending on the angle $\theta$, but not on $k$. In general,
condition~\eqref{eq:stiffmatrixlimitsrat} is related to the regularity of the mesh. 

For the two-dimensional Laplacian on non-uniform meshes we can show that~\eqref{eq:stiffmatrixlimitsrat} is satisfied provided 
there exists $\alpha_0>0$ so that
\begin{equation}
\label{eq:accute_angle_lim}
\alpha_0\le \alpha \le \frac{\pi}{2}-\alpha_0.
\end{equation}
Indeed, given an interior edge $\overline{P_i P_j}$ in the notation from Section~\ref{ssec:classresult} {\color{black}
and referring to Fig.~\ref{fig:triangle_formula} (right)} we have cf.~\eqref{eq:formula4triangle} 
{\color{black}
$$
\int_{\Delta P_i P_j P_k}\varphi_i \varphi_j = \frac{h_k^2}{24(\cot \theta_i + \cot \theta_j)} \le \frac{h_k^2}{48 \cot\left(\frac{\pi}{2}-\alpha_0\right)} = 
\frac{h_k^2}{48\tan \alpha_0},
$$}
where $h_k = |\overline{P_i P_j}|$.
{\color{black}
Since two adjacent triangles contribute to $\mat{M}_{ij}$, we have 
$$
\mat{M}_{ij} \le \frac{h_k^2}{24\tan \alpha_0}.
$$
}
Using~\eqref{eq:formula2triangle} and the notation in Figure~\ref{fig:triangle_formula} (left), we have
{\color{black}
$$
-\mat{A}_{ij} = \frac{\sin (\theta_1+\theta_2)}{2 \sin \theta_1 \sin \theta_2}
= \frac{1}{2}(\cot \theta_1+\cot \theta_2) \ge \cot \left(\frac{\pi}{2}-\alpha_0\right) = \tan \alpha_0.
$$
It follows that
$$
\mat{M}_{ij} \le \frac{h_k^2\tan\alpha_0}{24 \tan^2\alpha_0} \le \frac{h_k^2}{24 \tan^2\alpha_0}\:|\mat{A}_{ij}|.
$$
Hence,  $C_A = (24 \tan^2\alpha_0)^{-1}$.
}
\end{remark}



\section{Conclusions}
\label{sec:conclusion}
We have introduced a novel technique for proving the global sDMP for elliptic equations; this can take  two forms, depending on the presence 
or absence of the zeroth order term. This method does not assume that all the off-diagonal entries of the stiffness matrix are nonpositive.
We applied this {\color{black}technique} to the ${\mathcal P}_1$ discretization of elliptic equations, and showed that the global sDMP can hold  even in the absence of
the local sDMP is being satisfied everywhere. The method is also applied to prove -- in a fairly natural fashion -- that the sDMP holds  
for semilinear elliptic equations as well.
As formulated, we believe the new method can be applied to other classes of nonlinear elliptic equations, as well as $P_2$ and $Q_1$ discretizations.
The idea of local to global extension of DMPs via graph connectivity can be applied to other classes of problems, such as 
finite element discretizations of parabolic equations, and perhaps other types of discretizations of elliptic equations as well.

\appendix
\section{The proof of Theorem~\ref{thm:degenerate_triangle}}
\label{sec:appendix}
Following a strategy similar to the one in~\cite{MR2085400},
we use a hierarchical basis to compute an approximation to the matrix $\mat{S}=\mat{S}(\alpha)$, 
which will lead to showing that $\mat{S}^{-1}>\mat{0}$ for sufficiently small $\alpha>0$.
More precisely, we consider the coarse mesh on the domain
$E$ resulted from removing the points $M, N, P$ (refer to Figure~\ref{fig:patch_around_triangle}); 
the coarse basis functions we use are $\widetilde{\varphi}_B, \widetilde{\varphi}_C, \widetilde{\varphi}_A$, 
all of which are linear  on the entire triangle $\Delta ABC$. The alternative basis in $\op{V}$ is
$\widetilde{\op{B}} = 
\{\widetilde{\varphi}_B, \widetilde{\varphi}_C, \widetilde{\varphi}_A, \varphi_M, \varphi_N, \varphi_P\}$, and 
we denote by $\widetilde{\mat{S}}$ the stiffness matrix computed in the basis $\widetilde{\op{B}}$.

If $\mat{A}=\mat{A}(\alpha), \mat{B}=\mat{B}(\alpha)$ are nonsingular square matrices of the same size, we use the notation
$\mat{A} \approx \mat{B}$ if $\mat{A} = \mat{B}(\mat{I} + \mat{O}(\alpha))$ for sufficiently small $\alpha$,
where $\mat{O}(\alpha)$ is a matrix with entries of size $O(\alpha)$.
It is a simple exercise to see that this an equivalence relation which is  closed to matrix product and inversion, in the sense that
\begin{eqnarray*}
&&\mat{A}_1\approx \mat{A}_2\ \ \mathrm{and}\ \ \mat{B}_1\approx \mat{B}_2\ \ \mathrm{implies}\ \ 
\mat{A}_1\mat{B}_1\approx \mat{A}_2\mat{B}_2,\\
&&\mat{A}_1\approx \mat{A}_2\ \ \mathrm{implies}\ \ \mat{A}^{-1}_1\approx \mat{A}^{-1}_2.
\end{eqnarray*}
If $\mat{B}$ is independent of $\alpha$ and invertible, 
then $\mat{A} = \mat{B}+\mat{O}(\alpha)$ implies $\mat{A} \approx \mat{B}$, because 
$$\mat{A} = \mat{B}+\mat{O}(\alpha) = \mat{B}\left(\mat{I}+\mat{B}^{-1}\mat{O}(\alpha)\right) = 
\mat{B}\left(\mat{I}+\mat{O}(\alpha)\right).$$ 

\begin{lemma}
\label{lma:smallstiffnessmattilde}
The matrix $\widetilde{\mat{S}}$ has the block form
\begin{equation}
\label{eq:smallstiffnessmatblock}
\widetilde{\mat{S}} = 
\begin{bmatrix} \mat{A} & \mat{B} \\
\mat{B}^T & \mat{C}\end{bmatrix},
\end{equation}
with the individual blocks given by
\begin{equation}
\label{eq:smallstiffnessmatdetails1}
{\mat{A}} =
\begin{bmatrix} 4& -1& -1\\-1& 4& 0\\-1& 0& 4 \end{bmatrix},\ \ 
{\mat{B}} = 
\begin{bmatrix} {1}/{2}& 0& 0\\{1}/{2}& 0& 0\\-{1}/{2}& 0& 0 \end{bmatrix},\ \ 
{\mat{C}_0} = 
\begin{bmatrix} 3/2& -1& -1\\-1& 5/4& 1/4\\-1& 1/4& 5/4 \end{bmatrix}
\end{equation}
and 
\begin{equation}
\label{eq:Cform}
\mat{C} = \frac{1}{\alpha}(\mat{C}_0+\mat{O}(\alpha)). 
\end{equation}
\end{lemma}
We postpone the technical proof of Lemma~\ref{lma:smallstiffnessmattilde} to the end of the section.

\begin{figure}[!h]
\begin{center}
        \includegraphics[width=5in]{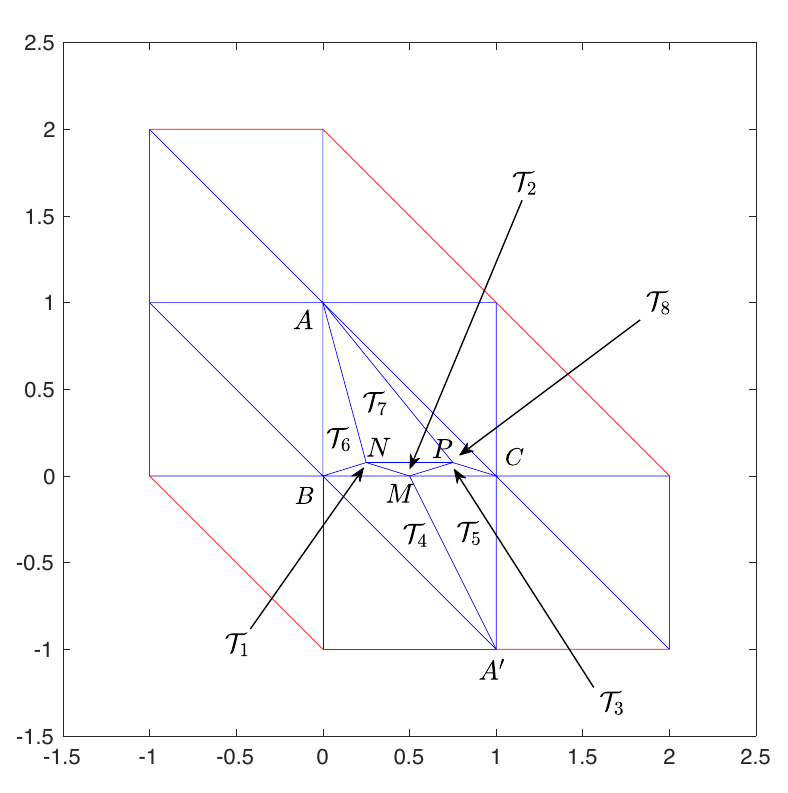}
\end{center}
\caption{Patch $E$ around triangle with degenerate mesh}
\label{fig:patch_around_triangle}
\end{figure}

\begin{lemma}
\label{lma:smallstiffnessmat}
The matrix ${\mat{S}}$ has the form 
\begin{equation}
\label{eq:smallstiffnessmat}
\mat{S} = \mat{E}\widetilde{\mat{S}}\mat{E}^T,
\end{equation}
where 
\begin{equation}
\label{eq:smallstiffnessmatBR}
\mat{E} =  \begin{bmatrix}\mat{I}& -\mat{R}\\\mat{0}& \mat{I}\end{bmatrix},\ \mathrm{with}\ 
\mat{R} =  \mat{R}_0 + \mat{O}(\alpha)\ \ \mathrm{and}\ \ 
\mat{R}_0 =\begin{bmatrix}1/2& 3/4& 1/4\\1/2& 1/4 & 3/4\\ 0& 0 & 0\end{bmatrix}.
\end{equation}
\end{lemma}
\begin{proof}Referring to Figures~\ref{fig:degenerate_triangle} and~\ref{fig:patch_around_triangle}, note that
\begin{eqnarray*}
\widetilde{\varphi}_B& =& {\varphi}_B + \widetilde{\varphi}_B(M) {\varphi}_M + \widetilde{\varphi}_B(N) {\varphi}_N + 
\widetilde{\varphi}_B(P) {\varphi}_P\\
& =& {\varphi}_B + \frac{1}{2} {\varphi}_M + \left(\frac{3}{4}+O(\alpha)\right) {\varphi}_N + 
\left(\frac{1}{4}+O(\alpha)\right) {\varphi}_P,
\end{eqnarray*}
where we used the fact that $N=N(\alpha) \to (1/4,0)$ as $\alpha\to 0$, 
showing that $\lim_{\alpha\to 0} \widetilde{\varphi}_B(N) = 3/4$. Due to the smoothness of the function 
$\alpha\mapsto \widetilde{\varphi}_B(N(\alpha))$, we conclude that  $\widetilde{\varphi}_B(N) = 3/4 + O(\alpha)$. 
Similarly, $\widetilde{\varphi}_B(P) = 1/4 + O(\alpha)$. Following the same 
line of arguments we have 
{\color{black}
\begin{eqnarray*}
\widetilde{\varphi}_C& =& {\varphi}_C + \widetilde{\varphi}_C(M) {\varphi}_M + \widetilde{\varphi}_C(N) {\varphi}_N + 
\widetilde{\varphi}_C(P) {\varphi}_P\\
& =& {\varphi}_C + \frac{1}{2} {\varphi}_M + \left(\frac{1}{4}+O(\alpha)\right) {\varphi}_N + 
\left(\frac{3}{4}+O(\alpha)\right) {\varphi}_P,
\end{eqnarray*}
}
and
\begin{eqnarray*}
\widetilde{\varphi}_A& =& {\varphi}_A + \widetilde{\varphi}_A(M) {\varphi}_M + \widetilde{\varphi}_A(N) {\varphi}_N + 
\widetilde{\varphi}_A(P) {\varphi}_P\\
& =& {\varphi}_A + O(\alpha) {\varphi}_N + O(\alpha) {\varphi}_P.
\end{eqnarray*}
If we redenote the basis functions (for the purpose of temporarily replacing the letter indices with numbers)
${\op{B}} = \{{\psi}_i\}_{i=1,\dots,6}$ and $\widetilde{\op{B}} = \{\widetilde{\psi}_i\}_{i=1,\dots,6}$,
then 
\begin{eqnarray*}
\psi_i = \sum_{j=1}^6 \mat{E}_{ij} \widetilde{\psi}_j,\ \ i=1,\dots,6,\ \ \mathrm{with}
\end{eqnarray*}
\begin{eqnarray*}
\mat{E} =  
\begin{bmatrix}
{\hspace{7pt}}1{\hspace{7pt}}& {\hspace{7pt}}0{\hspace{7pt}}& {\hspace{7pt}}0{\hspace{7pt}}& -\varphi_B(M)& -\varphi_B(N)& -\varphi_B(P)\\
0& 1& 0& -\varphi_C(M)& -\varphi_C(N)& -\varphi_C(P)\\
0& 0& 1& -\varphi_A(M)& -\varphi_A(N)& -\varphi_A(P)\\
0& 0& 0& 1& 0& 0\\
0& 0& 0& 0& 1& 0\\
0& 0& 0& 0& 0& 1
\end{bmatrix}
\end{eqnarray*}

Therefore,
\begin{eqnarray*}
\mat{S}_{ij} & =&  a_E(\psi_j,\psi_i) = 
a_E\left(\sum_{k=1}^6\mat{E}_{jk}\widetilde{\psi}_k,\sum_{k=1}^6\mat{E}_{il}\widetilde{\psi}_l\right)
=\sum_{k,l=1}^6\mat{E}_{jk} \mat{E}_{il}\: a_E\left(\widetilde{\psi}_k,\widetilde{\psi}_l\right)\\
& = & \sum_{k,l=1}^6\mat{E}_{jk} \mat{E}_{il} \widetilde{\mat{S}}_{lk}  = 
\left(\mat{E} \widetilde{\mat{S}}\mat{E}^T\right)_{ij},
\end{eqnarray*}
showing~\eqref{eq:smallstiffnessmat}.
\end{proof}

We can now proceed to prove of Theorem~\ref{thm:degenerate_triangle}.
\begin{proof}
Using~\eqref{eq:smallstiffnessmat} we get
\begin{equation}
\label{eq:smallstiffnessmatinv}
\mat{S}^{-1} = \mat{E}^{-T}\widetilde{\mat{S}}^{-1}\mat{E}^{-1}.
\end{equation}
For computing $\widetilde{\mat{S}}^{-1}$ we use the block inversion formula 
\begin{equation}
\label{eq:blockinverse}
\widetilde{\mat{S}}^{-1} = 
\begin{bmatrix} 
\mat{I} & \mat{0} \\
\mat{D}^T  & \mat{I}
\end{bmatrix}\ 
\begin{bmatrix} 
{\widehat{\mat{S}}}^{-1} & \mat{0} \\
\mat{0}  & {\mat{C}}^{-1} 
\end{bmatrix}\ 
\begin{bmatrix} 
\mat{I} & \mat{D} \\
\mat{0}  & \mat{I}
\end{bmatrix}
\end{equation}
with 
$\widehat{\mat{S}} = {\mat{A}} - 
{\mat{B}} {\mat{C}}^{-1}{\mat{B}}^T$ 
being the Schur complement of ${\mat{C}}$,  and
$\mat{D} = -{\mat{B}}{\mat{C}}^{-1}$.
Note that $\alpha \mat{C} \approx \mat{C_0}$, hence 
$\alpha^{-1}\mat{C}^{-1} \approx \mat{C_0}^{-1}$, showing that
\begin{eqnarray*}
\mat{C}^{-1} = \alpha\left(\mat{C_0}^{-1}+\mat{O}(\alpha)\right) = 
\alpha
\left(
\begin{bmatrix} 
6 &  4 &   4\\
4 & 7/2 &  5/2\\
4 & 5/2 & 7/2
\end{bmatrix}
+\mat{O}(\alpha)\right) .
\end{eqnarray*}
Since $\mat{B} = \mat{O}(1)$ and $\mat{C}^{-1}  = \mat{O}(\alpha)$, it follows that
\begin{eqnarray*}
\widehat{\mat{S}} = {\mat{A}} - {\mat{B}} {\mat{C}}^{-1}{\mat{B}}^T = {\mat{A}} + \mat{O}(\alpha),
\end{eqnarray*}
showing that $\widehat{\mat{S}} \approx {\mat{A}}$. Hence, 
\begin{eqnarray}
\label{eq:Ainv}
\widehat{\mat{S}}^{-1} \approx
\mat{A}^{-1} = 
\begin{bmatrix} 
2/7  &  1/14 &  1/14\\
1/14 & 15/56 &  1/56\\
1/14 &  1/56 & 15/56
\end{bmatrix} > \mat{0}.
\end{eqnarray}
Therefore, using~\eqref{eq:blockinverse} we get
\begin{eqnarray*}
\label{eq:blockinverseeplicit}
&&\widetilde{\mat{S}}^{-1} =
\begin{bmatrix} 
\mat{I} & \mat{0} \\
\mat{D}^T  & \mat{I}
\end{bmatrix}\ 
\begin{bmatrix} 
{\widehat{\mat{S}}}^{-1} & \mat{0} \\
\mat{0}  & {\mat{C}}^{-1} 
\end{bmatrix}\ 
\begin{bmatrix} 
\mat{I} & \mat{D} \\
\mat{0}  & \mat{I}
\end{bmatrix}
=
\begin{bmatrix} 
{\widehat{\mat{S}}}^{-1} & {\widehat{\mat{S}}}^{-1}\mat{D} \\
\mat{D}^T {\widehat{\mat{S}}}^{-1} &  {\mat{C}}^{-1} + \mat{D}^T {\widehat{\mat{S}}}^{-1}\mat{D}
\end{bmatrix}
\end{eqnarray*}
Note that
\begin{eqnarray*}
\mat{D} &=& -{\mat{B}}{\mat{C}}^{-1} = 
-\alpha\begin{bmatrix} {1}/{2}& 0& 0\\{1}/{2}& 0& 0\\-{1}/{2}& 0& 0 \end{bmatrix}
\left(\begin{bmatrix} 
6 &  4 &   4\\
4 & 7/2 &  5/2\\
4 & 5/2 & 7/2
\end{bmatrix}
+\mat{O}(\alpha)\right)\\
&=&
\alpha
\left(\begin{bmatrix}
-3& -2& -2\\-3& -2& -2\\3& 2& 2
\end{bmatrix}+\mat{O}(\alpha)\right) = \alpha \left(\mat{D}_0 + \mat{O}(\alpha)\right).
\end{eqnarray*}
Therefore
\begin{eqnarray*}
{\mat{C}}^{-1} + \mat{D}^T {\widehat{\mat{S}}}^{-1}\mat{D}  = {\alpha}(\mat{C}^{-1}_0+\mat{O}(\alpha)) + \mat{O}(\alpha^2)
={\alpha}(\mat{C}^{-1}_0+\mat{O}(\alpha)).
\end{eqnarray*}
Hence
\begin{eqnarray}
\label{eq:blockinverseeplicit2}
&&\widetilde{\mat{S}}^{-1} =
\begin{bmatrix} 
{{\mat{A}}}^{-1} + \mat{O}(\alpha) & \alpha {{\mat{A}}}^{-1}\left(\mat{D}_0 + \mat{O}(\alpha)\right) \\
\alpha \left(\mat{D}^T_0 + \mat{O}(\alpha)\right) {{\mat{A}}}^{-1} &  {\alpha}(\mat{C}^{-1}_0+\mat{O}(\alpha))
\end{bmatrix}
\end{eqnarray}
Putting together~\eqref{eq:smallstiffnessmatBR},~\eqref{eq:smallstiffnessmatinv}, and~\eqref{eq:blockinverseeplicit2}
we get
\begin{eqnarray*}
\mat{S}^{-1} 
&=&
%
%
%
%
%
\begin{bmatrix} 
{{\mat{A}}}^{-1} &
{{\mat{A}}}^{-1} \mat{R}_0\\
\mat{R}^T_0 {\mat{A}}^{-1}&
\mat{R}^T_0 {\mat{A}}^{-1} \mat{R}_0
\end{bmatrix} +\mat{O}(\alpha) = \mat{T}_0 + \mat{O}(\alpha).
\end{eqnarray*}
A direct computation shows that
\begin{eqnarray*}
\mat{A}^{-1} \mat{R}_0 = 
\begin{bmatrix} 
5/28 &  13/56 &  1/8\\
19/112 & 27/224 & 7/32\\
5/112& 13/224 & 1/32
\end{bmatrix}
\end{eqnarray*}
and 
\begin{eqnarray*}
\mat{R}^T_0\mat{A}^{-1} \mat{R}_0 = 
\begin{bmatrix} 
39/224 &  79/448 &  11/64\\
79/448 & 183/896 & 19/128\\
11/64 &   19/128 & 25/128
\end{bmatrix},
\end{eqnarray*}
showing, together with~\eqref{eq:Ainv}, 
that $\mat{T}_0 >\mat{0}$. Since $\mat{S}^{-1}  = \mat{T}_0  +\mat{O}(\alpha)$, we have 
$$
\lim_{\alpha\to 0} \mat{S}^{-1}  = \mat{T}_0.
$$
Note that $\mat{T}_0$ is singular, {\color{black} and has rank 3, hence is rank-3 deficient}. The computation has been validated numerically against stiffness matrix computations using 
standard finite element codes.
\end{proof}

We return to the proof of Lemma~\ref{lma:smallstiffnessmattilde}.
\begin{proof}
The fact that the ${\mat{A}}$ block has the form~\eqref{eq:smallstiffnessmatdetails1} is a simple 
consequence of the formulas~\eqref{eq:formula2triangle} and~\eqref{eq:formula1triangle}, and are well known for a uniform 
grid. Due to the fact  $\widetilde{\varphi}_A$, $\widetilde{\varphi}_B$, $\widetilde{\varphi}_C$ are linear on 
$\Delta ABC$, and $\varphi_N, \varphi_P$ vanish on $\partial (\Delta ABC)$, we have 
\begin{eqnarray*}
a_E(\widetilde{\varphi}_X, \varphi_Y) = 0,\ \ \forall X \in \{A, B, C\},\ \  Y \in \{N, P\}.
\end{eqnarray*}
(see also Lemma~5.5 in~\cite{MR2085400}), thus justifying all the six zero-entries in the matrix~${\mat{B}}$.
The remaining non-trivial entries are associated with  the set $\{\varphi_M, \varphi_N, \varphi_P\}$. 

{\underline {Entries related to ${\varphi}_N, {\varphi}_P$}:}
This case has been discussed in Lemma~5.6 in~\cite{MR2085400}; for completeness we review the computation,
which is simplified because we are interested primarily in the case when $0< \alpha \ll 1$. Refer to 
Figure~\ref{fig:patch_around_triangle} for notation.
We have
\begin{eqnarray*}
a_E(\varphi_N, \varphi_N) &= &\int_{\op{T}_1\cup \op{T}_2\cup \op{T}_6\cup\op{T}_7} |\nabla\varphi_N|^2 \\
&= &\frac{\sin(\pi-2 \alpha)}{2 \sin^2\alpha} + \frac{\sin\alpha}{2 \sin(\pi-2 \alpha) \sin\alpha} + 
\int_{\op{T}_6\cup\op{T}_7} |\nabla \varphi_N|^2\\
& = & \frac{\cos \alpha}{\sin\alpha} + \frac{1}{2 \sin(2 \alpha)} + O(1) = 
\frac{1}{\alpha}\left(\frac{5}{4} + O(\alpha) \right).
\end{eqnarray*}
Similarly,
\begin{eqnarray*}
a_E(\varphi_P, \varphi_P) &= & 
\frac{1}{\alpha}\left(\frac{5}{4} + O(\alpha) \right).
\end{eqnarray*}
Furthermore, if $\gamma$ denotes the angle $\sphericalangle NAP$, then
\begin{eqnarray*}
a_E(\varphi_N, \varphi_P)& = & -\frac{\sin(\gamma+\pi-2\alpha)}{2 \sin \gamma \sin(\pi-2\alpha)} 
 =  -\frac{\sin(2\alpha-\gamma)}{2 \sin \gamma \sin(2\alpha)} \\
&= &
-\frac{\cos \gamma}{2 \sin \gamma}
+\frac{\cos(2\alpha)}{2 \sin(2\alpha)} = \frac{1}{\alpha}\left(\frac{1}{4} + O(\alpha) \right)
\end{eqnarray*}
because $\lim_{\alpha\to 0} \gamma(\alpha) = \gamma_0\in (0,\pi/4)$, showing that the expression involving
$\gamma$ is bounded as $\alpha\to 0$.

{\underline {Entries related to ${\varphi}_M$}:}
The quantities $a_E(\varphi_M, \widetilde{\varphi}_X)$
with $X \in \{A, B, C\}$ cannot be computed using~\eqref{eq:formula2triangle}, because the latter are 
not nodal basis functions with respect to the finer mesh.
First note that $\nabla \widetilde{\varphi}_A = e_2$, showing that 
\begin{eqnarray*}
a_E(\varphi_M, \widetilde{\varphi}_A)& = & \int_{\op{T}_1\cup  \op{T}_2\cup \op{T}_3} \partial_y\varphi_M.
\end{eqnarray*}
Let $g:[0,1]\to E$ be the piecewise linear function whose graph is represented by 
the contour $BNPC$. Using Fubini's theorem and the fact that $\varphi_M(g(x)) = 0$, 
we obtain
\begin{eqnarray*}
\int_{\op{T}_1\cup  \op{T}_2\cup \op{T}_3} \partial_y\varphi_M & = &
\int_0^1 dx \int_0^{g(x)}\partial_y\varphi_M(x,y) = - \int_0^1 \varphi_M(x,0) dx = -\frac{1}{2},
\end{eqnarray*}
because $\varphi_M(x,0)$ is a one-dimensional, piecewise linear nodal basis function on $\overline{BC}$.
Therefore $\widetilde{{S}}_{34} = \widetilde{{S}}_{43}=-1/2$.
For similar reasons, if $A'=(1,-1)$ denotes the vertex on $\partial E$ 
directly below $C$ (see Figure~\ref{fig:patch_around_triangle}), and $\widetilde{\varphi}_{A'}$
is the coarse nodal basis function associated with $A'$, then
\begin{eqnarray*}
a_E(\varphi_M, \widetilde{\varphi}_{A'}) &=& \int_{\op{T}_4\cup \op{T}_5} 
\nabla {\varphi}_{M} \cdot \nabla \widetilde{\varphi}_{A'} = -\frac{1}{2}.
\end{eqnarray*}

Since $\mathrm{supp}(\varphi_M) = \bigcup_{i=1}^5 \op{T}_i 
\subseteq\mathrm{supp}(\widetilde{\varphi}_B)\cap \mathrm{supp}(\widetilde{\varphi}_C)$,
we have 
\begin{eqnarray*}
a_E(\varphi_M, \widetilde{\varphi}_X)& = & \int_{\op{T}_1\cup\dots\cup \op{T}_5} 
\nabla \varphi_M\cdot \nabla \widetilde{\varphi}_X,\ \ X\in \{B,C\}.
\end{eqnarray*}
Note that on $D_M^{(1)} = \op{T}_1\cup \op{T}_2 \cup \op{T}_3$ we have $\nabla \widetilde{\varphi}_C = e_1$, while on
$D_M^{(2)} = \op{T}_4\cup \op{T}_5$ we have  $\nabla \widetilde{\varphi}_B = -e_1$.
Hence, after using Fubini's theorem 
\begin{eqnarray}
\label{eq:Dm1}
&&\int_{D_M^{(1)}} \nabla \varphi_M \cdot \nabla \widetilde{\varphi}_C = 
\int_{D_M^{(1)}} \partial_x \varphi_M = \int_0^{\frac{1}{4}\tan \alpha } dy \int_{D_{M,y}^{(1)}} \partial_x  \varphi_M \ dx = 0,
\end{eqnarray}
where $D_{M,y}^{(1)}$  is the horizontal section of $D_{M}^{(1)}$ at level $y$, and we use the fact that $\varphi_M$
is zero at the end points of each horizontal section of $D_{M}^{(1)}$. Similarly,
\begin{eqnarray}
\label{eq:Dm2}
&& \int_{D_M^{(2)}} \nabla \varphi_M \cdot \nabla \widetilde{\varphi}_B = 
-\int_{D_M^{(2)}} \partial_x \varphi_M = -\int_{-1}^{0} dy \int_{D_{M,y}^{(2)}} \partial_x  \varphi_M \  dx = 0,
\end{eqnarray}
for the same reason. 
It remains that 
\begin{eqnarray}
\label{eq:MBC}
&& a_E(\varphi_M, \widetilde{\varphi}_B) =  \int_{D_M^{(1)}}
\nabla \varphi_M\cdot \nabla \widetilde{\varphi}_B,\ \ \ 
a_E(\varphi_M, \widetilde{\varphi}_C) =  \int_{D_M^{(2)}}
\nabla \varphi_M\cdot \nabla \widetilde{\varphi}_C.
\end{eqnarray}

On $\Delta ABC$ we have $\widetilde{\varphi}_B + \widetilde{\varphi}_C + \widetilde{\varphi}_A \equiv 1$, thus
(implicitly) this holds on $D_M^{(1)}$. So 
\begin{eqnarray*}
0 &= & \int_{D_M^{(1)}} \nabla \varphi_M \cdot \nabla (\widetilde{\varphi}_B + \widetilde{\varphi}_C + \widetilde{\varphi}_A)
\stackrel{\eqref{eq:Dm1}}{=} \int_{D_M^{(1)}} \nabla \varphi_M \cdot \nabla \widetilde{\varphi}_B + 
\int_{D_M^{(1)}}  \nabla \varphi_M \cdot \nabla \widetilde{\varphi}_A\\
& \stackrel{\eqref{eq:MBC}}{=} & a_E(\varphi_M, \widetilde{\varphi}_B) + a_E(\varphi_M, \widetilde{\varphi}_A),
\end{eqnarray*}
showing that 
\begin{eqnarray*}
a_E(\varphi_M, \widetilde{\varphi}_B) = \frac{1}{2}.
\end{eqnarray*}
On $\Delta A'BC$ we have $\widetilde{\varphi}_B + \widetilde{\varphi}_C + \widetilde{\varphi}_{A'} \equiv 1$, thus
(implicitly) this holds on $D_M^{(2)}$. So
\begin{eqnarray*}
0 &= & \int_{D_M^{(2)}} \nabla \varphi_M \cdot  \nabla (\widetilde{\varphi}_B + \widetilde{\varphi}_C + \widetilde{\varphi}_{A'})
\stackrel{\eqref{eq:Dm2}}{=} \int_{D_M^{(2)}} \nabla \varphi_M \cdot \nabla \widetilde{\varphi}_C + 
\int_{D_M^{(2)}}  \nabla \varphi_M \cdot \nabla \widetilde{\varphi}_{A'}\\
& \stackrel{\eqref{eq:MBC}}{=} & a_E(\varphi_M, \widetilde{\varphi}_C) + a_E(\varphi_M, \widetilde{\varphi}_{A'}),
\end{eqnarray*}
showing that 
\begin{eqnarray*}
a_E(\varphi_M, \widetilde{\varphi}_C) = \frac{1}{2},
\end{eqnarray*}
as well. Therefore, 
\begin{eqnarray*}
&&
\widetilde{S}_{14} = \widetilde{{S}}_{41}= \widetilde{{S}}_{24} = \widetilde{{S}}_{42}=1/2,
\end{eqnarray*}
which concludes the computation of the block ${\mat{B}}$ (recall that the other entries are $0$).
All these results can also be obtained by classical trigonometric arguments. Furthermore, using
the formulas~\eqref{eq:formula2triangle} and~\eqref{eq:formula1triangle} we obtain
\begin{eqnarray*}
&&a_E(\varphi_M, \varphi_N) = a_E(\varphi_M, \varphi_P) =  -\frac{\sin (2 \alpha)}{2 \sin^2\alpha}=
   -\frac{\cos \alpha }{ \sin\alpha} =-\frac{1}{\alpha} (1 + O(\alpha)).
\end{eqnarray*}
The last entry needed is
\begin{eqnarray*}
a_E(\varphi_M, \varphi_M) &=& \int_{\op{T}_1\cup \dots \cup \op{T}_5}|\nabla \varphi_M|^2 \\
&= & 
  \frac{1}{\sin(\pi-2\alpha)} + 
  \frac{\sin(\pi-2\alpha)}{2\sin^2 \alpha} + \overbrace{\int_{\op{T}_4\cup \op{T}_5}|\nabla \varphi_M|^2}^{O(1)}\\
  &=& \frac{3}{2\alpha} (1+O(\alpha)),
\end{eqnarray*}
which concludes the computation of~$\widetilde{\mat{S}}$.
\end{proof}

{\color{black}
\section{Limits of discrete Green's functions on some degenerate meshes}
\label{sec:appendix2}
In this section we revisit a class of two-dimensional, triangular meshes analyzed in Section~6 of~\cite{MR2085400}, 
for which the $\op{P}_1$-finite element solution of the Poisson equation with homogeneous Dirichlet boundary conditions
satisfies the DMP, while potentially violating the angle condition in many places.
The main purpose here is to expose the behavior of  the discrete Green's function  as the mesh becomes 
degenerate. The examples in this section are easier to analyze than the one in Theorem~\ref{thm:degenerate_triangle}.

\paragraph{\bf The case of a single interior node.}
The simplest nontrivial example of a finite element mesh is that of a triangular domain $D = \Delta ABC$ with a single
vertex $P$ in the interior. We choose $P$ so that the triangle $\Delta PBC$ is isosceles with base angles 
equal to $\theta$, as shown in Fig.~\ref{fig:triangle_degenrate_mesh} (left), although this is not essential.
With only one interior mesh vertex, the stiffness matrix is just the real number
{\color{black}$$\mat{A}(\theta) = a_D(\varphi_P,\varphi_P) = \int_{D} |\nabla \varphi_P|^2,$$}
where $\varphi_P$ is the nodal basis function associated with $P$.

We are interested in the behavior of the discrete Green's function as $\theta\to 0$. Cf~\eqref{eq:formula1triangle},
{\color{black}
$$
\mat{A}(\theta) \ge \int_{\op{T}_1} |\nabla \varphi_P|^2  = \frac{\sin(\pi-2\theta)}{2\sin^2 \theta} = \frac{2\sin\theta\cos\theta}{2\sin^2 \theta}
 = \cot \theta.
$$
}
Hence, 
\begin{equation}
\label{eq:convAPzero}
\lim_{\theta\to 0}(\mat{A}(\theta))^{-1} \le  \lim_{\theta\to 0}\tan \theta = 0.
\end{equation}
This shows that the discrete Green's function, which is represented by  a number, 
converges to the only singular $1\times 1$ matrix, namely the number 0.
Also note that the mass matrix of the mesh stays bounded as $\theta\to 0$.
Therefore, given a fixed right hand side $f$ for the Poisson equation on $D$ with the mesh described above,
the finite element solution $u(\theta)$ will converge to $0$ as $\theta \to 0$.
We remark that the limiting mesh, shown Fig.~\ref{fig:triangle_degenrate_mesh} (right), is a valid triangulation of $D$ with
no interior nodes. Note that the Poisson problem with homogeneous Dirichlet boundary conditions can also
be formulated on the associated finite element space, which contains only the function 0. Hence, the solution will be zero,
which is consistent with the limit of the solutions $u(\theta)$.

\begin{figure}[!h]
\begin{center}
        \includegraphics[width=5.0in]{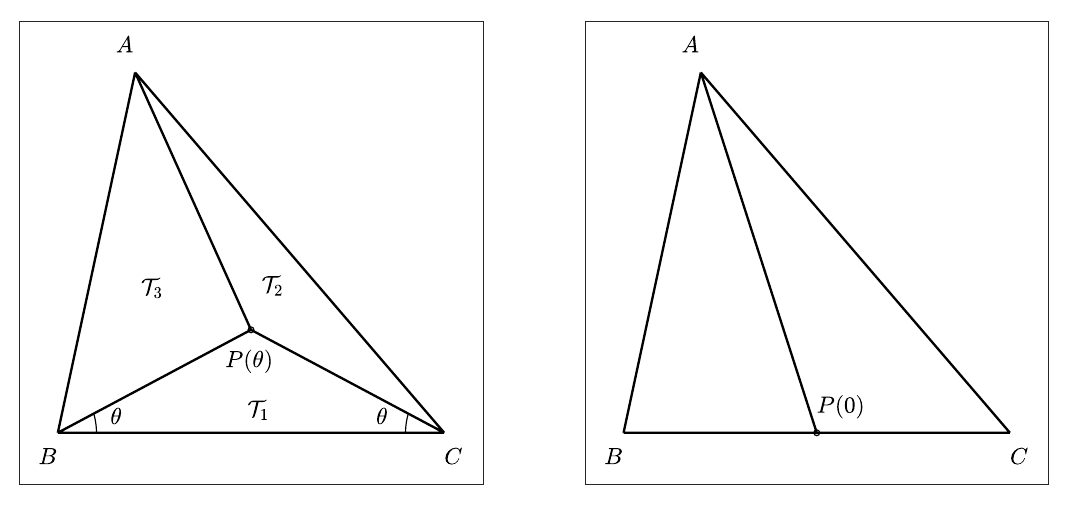}
\end{center}
\caption{As $\theta \to 0$ the mesh on the left converges to the mesh in the right image. When $\Delta A B C$ is
an interior triangle of another triangular mesh, the partition on the right will not be a valid mesh, since the triangle
below $\Delta A B C$ -- not pictured -- is not divided.}
\label{fig:triangle_degenrate_mesh}
\end{figure}

\paragraph{\bf An embedded divided triangle.} For the second example consider 
a triangular mesh $\op{T}_h$ on a polygonal domain $D$, and we subdivide
one interior triangle as in Fig.~\ref{fig:triangle_degenrate_mesh} (left), which results in a refined mesh 
$\tilde{\op{T}}_h = \tilde{\op{T}}_h(\theta)$.
If the discrete Green's function is positive (or just nonnegative) on $\op{T}_h$, 
it is shown in Section 6 of~\cite{MR2085400} that the same holds for $\tilde{\op{T}}_h$, regardless of the value of $\theta>0$. 
For $\theta = 0$, the subdivided triangle is shown in Fig.~\ref{fig:triangle_degenrate_mesh} (right). The resulting partition
of $D$ corresponding to $\theta = 0$, with $P(0)$ being the midpoint between $B$ and $C$, is not a valid triangulation,
because the triangle  in $\op{T}_h$ below  $\Delta A B C$ is not subdivided. A similar invalid partition is depicted 
in~\cite{MR2373954}, Fig.~3.12 (right), page 80.

The question we answer here is what happens to the discrete Green's function on $\tilde{\op{T}}_h(\theta)$ 
as $\theta\to 0$. We will describe this
limit in terms of the discrete Green's function on ${\op{T}}_h$, and we show it is represented by a rank-1 deficient matrix.
We denote by $\op{B} = \{\varphi_1,\dots,\varphi_n\}$ the nodal basis associated with the interior vertices of $\op{T}_h$, 
and let ${V}_0^h = \mathrm{span}(\op{B})$. Assume that $\varphi_{n+1}$ is the nodal basis function in 
$\tilde{\op{T}}_h$ associated with $P = P(\theta)$ 
(this was called $\varphi_P$ from the previous example), and $\varphi_{n-1}, \varphi_{n}$
correspond to $B$ and $C$, respectively. The set 
$\tilde{\op{B}} = \op{B} \cup \{\varphi_{n+1}\}$ forms a hierarchical basis for the finite element space $\tilde{V}_0^h$ on 
$\tilde{\op{T}}_h$. 
Cf.~\cite{MR2085400}, we have  $a_D(\varphi_i,\varphi_{n+1}) = 0$
for $i=1, \dots, n$, which leads to a decoupling of the linear system representing the Poisson equation when formulated 
in the hierarchical basis $\tilde{\op{B}}$. For $f\in (\tilde{V}_0^h)^*$, the system reads:
\begin{equation}
\label{eq:matpoisson}
\begin{bmatrix}
\mat{A} & 0\\
0 & \alpha
\end{bmatrix} 
\begin{bmatrix}
\mat{x}\\
x_{n+1}
\end{bmatrix} 
= 
\begin{bmatrix}
\mat{b}\\
b_{n+1}
\end{bmatrix},
\end{equation}
where $\mat{A}$ is the stiffness matrix of the problem in the basis $\op{B}$ (hence on $V_0^h$),
$\alpha = a_D(\varphi_{n+1},\varphi_{n+1})$,
with ${\mat{b}}\in \R^n$ given by
${\mat{b}}_i = \innprd{\varphi_{i}}{f}$ for $1\le i \le n$, and $b_{n+1} = \innprd{\varphi_{n+1}}{f}$.
The solution 
of~\eqref{eq:matpoisson} is given by
\begin{equation}
\label{eq:solpoisson}
u =  u(\theta) = \sum_{i=1}^{n+1} x_i \varphi_i,\ \ \mat{x} = \mat{A}^{-1} \mat{b},\ \ x_{n+1} =  {b}_{n+1}/\alpha.
\end{equation}
If $f$ is fixed and $\theta \to 0$, it follows from~\eqref{eq:convAPzero} that $x_{n+1} \to 0$. Hence, 
\begin{equation}
\label{eq:u0}
u_0 = \lim_{\theta\to 0} u(\theta) = \sum_{i=1}^{n} x_i \varphi_i \in V_h^0.
\end{equation}
It follows that 
\begin{equation}
\label{eq:u0mean}
u_0(P(0)) = \frac{1}{2}(u_0(B) + u_0(C)),
\end{equation}
showing that the limit as $\theta\to 0$ of all the finite element solutions on $\tilde{\op{T}}$ lie in a 
subspace of co-dimension 1. In other words, the degenerate partition can be considered ``harmless'', as the
limiting solution $u_0$ is simply the solution on the original mesh $\op{T}_h$.

In order to describe the matrix representation $\tilde{\mat{G}}(\theta)$ of the discrete Green's function 
on $\tilde{\op{T}}$ and its limit as $\theta\to 0$, 
we let $f^{(i)}$ be the Dirac impulse forcing at the $i^{\mathrm{th}}$ node, for $i\le n$.
Then in~\eqref{eq:solpoisson}  we have $\mat{b}=\mat{e}_i$ ($i^{\mathrm{th}}$ unit vector), and $b_{n+1} = 0$. 
Cf.~\eqref{eq:u0} the limiting solution 
$u_0^{(i)}$ lies in $V_h^0$, and satisfies the same equation as the discrete Green's function 
on $\op{T}_h$. Hence,~\eqref{eq:u0mean} implies that the vector representation  of $u_0^{(i)}$ in the 
basis $\tilde{\op{B}}$ is
$$
(\tilde{\mat{g}}^{(i)})^T = [(\mat{g}^{(i)})^T, \frac{1}{2}(\mat{g}^{(i)}_{n-1}+\mat{g}^{(i)}_{n})],
$$
where $\mat{g}^{(i)}$ is the $i^{\mathrm{th}}$ column of the discrete Green's function on $\op{T}_h$.

If $f^{(n+1)}$ is the Dirac impulse forcing at $P(\theta)$, then the corresponding right-hand side 
in~\eqref{eq:solpoisson} satisfies,
as $\theta\to 0$
$$\mat{b}=\frac{1}{2}\left(\mat{e}_{n-1} + \mat{e}_{n}\right),\ \ b_{n+1} = 1.$$
By~\eqref{eq:u0}, its vector representation
is 
$$
(\tilde{\mat{g}}^{(n+1)})^T = [\frac{1}{2}(\mat{g}^{(n-1)}+\mat{g}^{(n)})^T, \frac{1}{2}(\mat{g}^{(n)}_{n-1}+\mat{g}^{(n-1)}_{n})].
$$
Hence, if we denote by the matrix representation in $\tilde{\op{B}}$ of the limit of the discrete function as $\theta\to 0$
is 
$$
\tilde{\mat{G}}_0 = \lim_{\theta\to 0} \tilde{\mat{G}}(\theta) = 
\begin{bmatrix}
\mat{G}& \overline{g}\\
\overline{\mat{g}}^T& \tilde{g},
\end{bmatrix} \in \R^{(n+1)\times (n+1)}.
$$
where $\mat{G}$ is the representation of the discrete Green's function on the original mesh $\op{T}_h$, and
$$
\overline{\mat{g}} = \frac{1}{2}(\mat{g}^{(n-1)}+\mat{g}^{(n)}),\ \ \tilde{g} = \frac{1}{2}(\mat{g}^{(n-1)}_{n-1}+\mat{g}^{(n)}_{n}).
$$
We used the equality  $\mat{g}^{(n-1)}_{n} = \mat{g}^{(n)}_{n-1}$ to compute the last diagonal entry $\tilde{g}$, 
which follows from the symmetry of  the discrete Green's function $\mat{G}$.
Hence, the last column of $\tilde{\mat{G}}_0$ is the average
of the previous two, showing $\tilde{\mat{G}}_0$ has rank $n$.
By contrast, the matrix $\mat{T}_0$ in Theorem~\ref{thm:degenerate_triangle} has rank-3 deficiency.
}

{\color{black}
\section*{Acknowledgement} The authors sincerely thank the anonymous reviewers for their careful 
reading of the manuscript and for their valuable comments and constructive suggestions, 
which helped improve the quality of this work.}

\bibliographystyle{amsplain}
\bibliography{DMP}

\end{document}